\documentclass[11pt]{article}
\pdfoutput=1
\usepackage{amsbsy,amsmath,amsthm,amssymb,graphicx,verbatim}
\usepackage{setspace, float, subcaption,xcolor}

\usepackage{multirow}

\usepackage{geometry}
\usepackage{algorithm}
\usepackage{algorithmic}

\usepackage{verbatim} 

\newtheorem{theorem}{Theorem}
\newtheorem{corollary}{Corollary}
\newtheorem{lemma}{Lemma}
\newtheorem{proposition}{Proposition}
\theoremstyle{remark}
\newtheorem{remark}{Remark}
 
\newtheorem{assumption}{Assumption}

\newfont{\msbm}{msbm10 at 11pt}

\RequirePackage[authoryear,round]{natbib}
\RequirePackage[colorlinks,allcolors=black]{hyperref}
\RequirePackage{graphicx}

\usepackage{afterpage}
\usepackage{mathrsfs}
\RequirePackage{amsthm,amsmath,amsfonts,amssymb}
\usepackage{stackengine}
\usepackage{booktabs}
\def\SetLeft{\scalebox{2}{\{}}
\def\SetRight{\scalebox{2}{\}}}
\def\SetLine{\big{|}}

\usepackage{stackengine}
\stackMath
\def\sss{\scriptscriptstyle}
\setstackgap{L}{8pt}

\newcommand{\StackDem}[2]{\stackunder{#1}{\sss (#2)}}
\def\Cov{\text{Cov}}

\begin{document}
\onehalfspacing

\title{Inference Optimal Long Run Variance Estimation with Lugsail Kernels}

\author{Rebecca P.\ Kurtz-Garcia \\ Department of Mathematical Science and Program of Statistical \& Data Sciences \\ Smith College \\ {\tt rkurtzgarcia@smith.edu} \and James M. Flegal \\ Department of Statistics \\ University of California, Riverside \\ {\tt jflegal@ucr.edu}}   

\date{May 2026}


\maketitle

\begin{abstract}
For datasets with unknown but stationary serial dependence, a robust long run variance estimator is essential to handle diverse scenarios. Spectral variance estimators are commonly used but tend to exhibit significant negative bias in the presence of positive correlation. To overcome this, zero lugsail estimators have been introduced, offering zero asymptotic bias regardless of the correlation structure. However, there are currently no guidelines for selecting the optimal bandwidth for lugsail estimators, a critical component in the estimation process. We propose an inference optimal bandwidth rule for lugsail estimators, based on nonstandard fixed-smoothing limiting distributions developed in our study. This approach significantly improves bias correction, accounts for variability, and provides an estimator optimized for robust inference. Our theoretical findings are supported by a simulation study. 

\end{abstract}

\newpage

\tableofcontents

\newpage
\section{Introduction} \label{sec:intro}

Random vector processes arise in time series analysis, econometrics, spectral analysis, and steady-state simulation. Although the applications, assumptions, and priorities of these realms differ, each field works with stationary correlated multivariate data which may contain serial correlation of an unknown form. Inference procedures in these settings critically rely on estimation of the long run variance (LRV), denoted by 
\[
\Omega = \sum_{h = -\infty}^\infty \Gamma(h), 
\]
where $\Gamma(h)$ is the lag-$h$ autocovariance function.  

From a statistical perspective, LRV estimation is closely tied to spectral density estimation at frequency zero \citep{priestley1981spectral, hannan:1970, brockwell2009time}.  In econometrics, the LRV arrises in heteroskedastic and autocorrelation consistent covariance matrix estimation \citep{andrews1991heteroskedasticity, newey1987simple, muller2007theory}.  In steady-state and Markov chain Monte Carlo simulations, $\Omega$ is the asymptotic covariance in limit theorems of Monte Carlo estimators \citep{glyn:whit:1992, chan:yau:2017, vats2019multivariate}.


Across these settings, several LRV estimators have been proposed, with the spectral variance (SV) estimator being the most common. The SV estimator applies a weighting scheme to the autocovariance matrices using a kernel function $\kappa^*$ and a bandwidth parameter $b \in [0,1]$. Suppose $T$ is the sample size and $\hat{\Gamma}(h)$ is the sample lag-$h$ autocovariance, then the SV estimator is given by
\begin{equation}
    \hat{\Omega}_T = \sum_{h = -(T-1)}^{T-1} \kappa^* \left (\frac{h}{bT} \right ) \hat{\Gamma}(h) \label{eq:SV_est} \; .
\end{equation}
The bandwidth parameter controls the proportion of autocovariance matrices given non-zero weight, which largely influences the estimator's finite sample properties \citep{hirukawa2023robust}.  

The kernel for an SV estimator is typically from the class of functions
\begin{multline*}
    \mathscr{K}_1 =   \SetLeft \kappa^*: \mathbb{R} \rightarrow [-1, 1] \SetLine \kappa^*(x) = \kappa^*(-x), \kappa^*(0) =1, c_2< \infty, \int_0^\infty \kappa^*(x) x dx < \infty,\\
    \kappa^*\text{ is piece wise smooth}, K^*(\omega)\geq 0, \forall \omega>0 \SetRight,
\end{multline*}
where $K^*(\omega) = \frac{1}{2 \pi} \int^\infty_{-\infty} \kappa^*(u) \exp(-iu\omega) du$ is the Fourier transformation of the kernel $\kappa^*$ and $c_i = \int^\infty_{-\infty} \left( \kappa^*(u)\right )^i du$. Bartlett, Parzen, and quadratic spectral kernels belong to $\mathscr{K}_1$, which is similar to the second class of kernels defined in \cite{andrews1991heteroskedasticity}.  SV estimators formed using kernels in $\mathscr{K}_1$ are known to be negatively biased in the presence of positive correlation \citep{andrews1991heteroskedasticity,muller2014hac,ng1996exact}. Moreover, the bias grows as the underlying correlation increases. 

Multiple lines of research investigate LRV estimators which exhibit smaller asymptotic bias compared to their counterparts in $\mathscr{K}_1$.  Reducing bias is particularly important when the underlying correlation is high, as it allows for more effective data collection and analysis at higher frequencies.  For example, flat-top estimators \citep{politis1995bias, politis2011higher} `flatten' the kernel around the origin, while Jackknife corrections construct a linear combination of kernels \citep{efron1982jackknife, ding:alex:2015, hardle2003note, hardle1990applied}.

Our focus is on lugsail kernels of \cite{vats2022lugsail}, which encompasses both the Bartlett flat-top estimator and the standard Jackknife correction as special cases.  The zero lugsail kernel is particularly notable as it eliminates the first-order asymptotic bias, making it a flexible and powerful option in practice.  Although many bandwidth selection rules for SV estimators exist, these typically rely on first-order approximations of asymptotic bias and hence are not applicable here. As a result, practitioners are left with only ad hoc methods of selecting a bandwidth for zero lugsail kernels. 

Once obtained, a LRV estimator is often used to calculate a test statistic with $d$ parameters, say $F_T$, to conduct inference that typically relies on a familiar limiting distribution. Unfortunately, the sample size $T$ often needs to be impractically large for the familiar limiting distribution to accurately capture the behavior of $F_T$. When the sample size is not large enough, the observed error rate can be far from the prescribed rate yielding a statistically invalid test \citep[see e.g.][]{muller2014hac, lazarus2018har}. 

This paper addresses issues inherent to LRV estimation and proposes a robust inference procedure, focusing on settings with high correlation. We concentrate on the zero lugsail kernels to address bias and obtain an alternative limiting distribution for $F_T$ to address variability. We further propose an adaptive bandwidth rule for the zero lugsail kernels that is designed for inference, leveraging both statistical and mathematical rationale for optimization.

Our findings yield three significant contributions. Firstly, we derive expressions for higher-order finite sampling bias terms associated with the LRV estimation.   Our approximation of higher-order bias terms incorporates essential information for bandwidth selection, ensuring the rule's versatility.

Secondly, we derive an alternative fixed-smoothing limiting distribution applicable to the lugsail estimators, using a similar framework as \cite{kiefer2005new, sun2014let}. The principles introduced here establish the groundwork for considering bandwidth rules for other estimators sharing properties with the lugsail, such as flat-top estimators, bias-adjusted estimators, and Jackknifed area estimators \citep{politis1995bias, li2023correcting, ding:alex:2015}. 

Building on the theory, we thirdly propose a novel inference-based bandwidth rule tailored for the zero lugsail estimator. The chosen bandwidth ensures that the test statistic is well approximated by its alternative limiting distribution, maintaining the observed Type 1 error rate in the neighborhood of $\alpha$. This is accomplished by allowing the tolerance bounds to contract in settings with greater variability, thereby prioritizing precision. This contribution builds on recent work on loss functions and bandwidth rules designed for inference \citep[see e.g.][]{lazarus2021size, lazarus2018har, sun2011robust, sun2013heteroskedasticity,sun2014let, sun2008optimal}.

We further conduct a simulation study demonstrating the effectiveness of the proposed methods.  The simulations compare Type 1 and Type 2 error rates for various sample sizes, error correlations, and smoothing settings.  The results show the proposed bandwidth rule typically produces Type 1 error rates closest to the nominal level, particularly for higher correlations.  Moreover, the proposed rule demonstrates robustness in more complex error structures, maintaining near-optimal performance even when not the best performing choice. Overall, the zero lugsail estimator with our bandwidth procedure yields smaller size distortions and substantially improves Type 1 error rates.  

Other LRV estimators have been proposed for hypothesis testing in models with errors that have serial correlation of unknown form, including one by \cite{kiefer2000simple} that eliminates the need for a tuning parameter. There are also other transformed versions of the SV estimator, such as the series \citep{lazarus2018har, sun2013heteroskedasticity}, steep origin \citep{phillips2006spectral}, and difference-based estimators \citep{bai2024difference, chan2022optimal}. The key advantage of the methods proposed in this paper is the improved testing performance in high-correlation settings, achieved with minimal increases to computational burden. The general form of the lugsail estimator also allows for adaptation to incorporate other competitive estimators.

The remainder of this paper is organized as follows. Section~\ref{sec:model} establishes assumptions and the problem space. Section~\ref{sec:estimators} defines relevant estimators and key values. Section~\ref{sec:limiting} derives the alternative fixed-smoothing liming distribution and expressions for the testing error rates that include high-order terms. Section~\ref{sec:bandwidth} establishes an optimal bandwidth rule. Lastly, Section~\ref{sec:simulation} concludes with a simulation study to demonstrate the effectiveness of our proposed methods.  Proofs of the main results are provided in the Appendix.

\section{Model, Assumptions, and Test} \label{sec:model}

Let $\{w_t\}$ be an stationary sequence of random vectors with mean zero and finite variance $\Omega_w$ that is possibly heteroskedastic. For $t = 1, \dots, T$, define
\begin{align*}
    x_t' = & (x_{t,1}, \dots, x_{t,p}) &  \theta' = & (\theta_1,  \dots, \theta_p)\\
    X' = & [x_1', x_2', \dots, x_T'] & Y' =&[ y_1, \dots, y_T] 
\end{align*}
where $\theta  \in \Theta \subseteq \mathbb{R}^p$ is a vector of parameters to be estimated, $X$ is a observed $(T \times p)$ matrix, and $\{y_t\}$ is a sequence of observed random variables of the following form
\begin{equation}
    y_t = x_t' \theta + w_t .
    \label{eq:econ_model}
\end{equation}
For simplicity we restrict our attention to linear relations of $\theta$, however, results can be amended to suit alternative structures. Let $R$ be a $d \times p$ real valued matrix, $d \geq p$, and $\theta_0$ be the true value. We are interested in testing the following relationship,

\begin{equation}
    \begin{matrix}
    H_0:& R \theta =  R \theta_0    \\
    H_A:&  R \theta \neq R \theta_0.
    \end{matrix}
    \label{eq:hypotheses_linear}
\end{equation}

We estimate $\theta$ via a standard generalized method of moments (GMM) procedure, which is valid for a wide variety of applications \citep{hansen1982large}. Let $v_t = (y_t, x_t')$, the vector of observed values at time $t$. Assume the following moment conditions, 
\begin{equation*}
    E[f(v_t; \theta_0)] = 0
    \label{eq:moments}
\end{equation*}
\noindent where $t = (1, \dots, T)$, $f(v_t; \theta)$ is a $m \times 1$ vector of twice differentiable continuous real valued functions, $m \geq p$, and $E[f(v_t; \theta)]$ has rank $p$. The GMM estimator of $\theta$ is defined as 
\begin{equation*}
    \hat{\theta}_T = \underset{\theta \in \Theta}{\text{argmin}}\frac{1}{T} g_T(\theta)' W_T g_T(\theta) 
\end{equation*}
\noindent where $g_t(\theta) = \frac{1}{T}\sum_{j=1}^t f(v_j; \theta)$ and $W_T$ is a $m \times m$ weighting matrix which may be dependent on the sample. Define $ G_t(\theta) = \frac{\partial g_t(\theta)}{\partial \theta'}$.

We establish the following set of assumptions, which are common under the fixed-smoothing asymptotic framework \citep{ kiefer2005new, lazarus2021size, sun2014let, zhang2013fixed, velasco2001edgeworth}. Let $\overset{p}{\rightarrow}$ and $\overset{d}{\rightarrow}$ indicate convergence in probability and convergence in distribution, respectively. 

Assumptions~\ref{assumption:1}-\ref{assumption:3} are standard in the literature.  

\begin{assumption}
    Suppose $\hat{\theta}_T \overset{p}{\rightarrow} \theta_0$  and $\theta_0$ is an interior point of $\Theta$.
    \label{assumption:1}
\end{assumption}

\begin{assumption}
    Let $ G_{\lfloor rT \rfloor}(\tilde{\theta}_T) \overset{p}{\rightarrow} r G_0$  uniformly in $ r\in [0,1]$ for any $\tilde{\theta}_T$ whose elements are between the corresponding elements of $\hat{\theta}_T$ and $\theta_0$, and where $G_0$ is a $m \times p$ matrix with full column rank. 
    \label{assumption:2}
\end{assumption}

\begin{assumption}
    Suppose $W_T$ is positive semi-definite (psd) and $W_T \overset{p}{\rightarrow} W_\infty$ where $W_\infty$ is a matrix of constants and $G_0' W_\infty G_0$ is positive definite.
    \label{assumption:3}
\end{assumption}

For ease of notation, let $u_t = -R \left(G_0' W_\infty G_0 \right)^{-1}G_0 W_\infty f(v_t;\theta_0)$ then the LRV of interest is
        \begin{align*}
        \Omega = \sum_{j = -\infty}^\infty E[u_tu_{t-j}'] .
        \end{align*}
Since $\Omega$ is positive definite, there is exactly one positive definite matrix $\Lambda$ such that $\Omega = \Lambda\Lambda'$.  Assumption~\ref{assumption:4} is required for the functional central limit theorem (FCLT), and is less strict than in other settings. For example, \cite{andrews1991heteroskedasticity} requires that $u_t$ is a zero mean fourth-order $\alpha$-mixing process which implies Assumption~\ref{assumption:4}. See \cite{phillips1986multiple, sun2014let, tanaka1996time} for a larger discussion on sufficient conditions under various settings. 

\begin{assumption}
    Let $T^{-1/2} \sum_{t=1}^{\lfloor rT \rfloor} u_t  \overset{d}{\rightarrow} \Lambda B_d(r) $ where $B_d(r)$ is a standard $d$-dimensional Brownian motion. 
\label{assumption:4}
\end{assumption}

Assumption~\ref{assumption:5} is similar to \cite{lazarus2021size, sun2014let, zhang2013fixed,velasco2001edgeworth} and provides conditions for deriving the bias of LRV estimators and obtaining expressions for the distribution of the test statistics under fixed-smoothing assumptions.

\begin{assumption} Suppose (i) the sequence $u_t$ is a stationary Gaussian process; (ii) for any $c \in \mathbb{R}^d$, the spectral density of $c' u_t$ is bounded above and away from zero in a neighborhood around the origin; and (iii) for some $r \in [0, 2 + \zeta]$ and $\zeta >0$, 
    \begin{equation}
        h_r :=  \sum_{s = -\infty}^\infty |s|^r |\Gamma(s)|<\infty.
        \label{eq:generalized_deriv}
    \end{equation}
    \label{assumption:5}
\end{assumption}

Under Assumptions~\ref{assumption:1}-\ref{assumption:4}, the estimated parameters converge to a normal distribution as seen in Lemma~\ref{lemma:wald_limiting}.  The proof can be found in the appendix.   

\begin{lemma}
    Under Assumptions~\ref{assumption:1}-\ref{assumption:4}, $T^{1/2}R (\hat{\theta}_T - \theta_0) \overset{d}{\rightarrow}  \Lambda B_p(1) \sim  N(0, \Omega)$.
    \label{lemma:wald_limiting}
\end{lemma}

Under this construction the Wald test statistics are
\begin{align}
    F_T = & \left [\sqrt{T} R (\hat{\theta}_T - \theta_0)\right ]' \hat{\Omega}_T^{-1}\left [\sqrt{T} R (\hat{\theta}_T - \theta_0)\right ]/d \label{eq:wald_limiting}\\
    t_T = & \left [\sqrt{T} R (\hat{\theta}_T - \theta_0)\right ]\hat{\Omega}_T^{-1/2} \nonumber
\end{align}
where $(\hat{\theta}_T, \hat{\Omega}_T)$ are consistent estimates of $(\theta, \Omega)$. With our SV estimators defined at \eqref{eq:SV_est}, the estimator for $\Omega$ takes the form 
$$\hat{\Omega}_T = \frac{1}{T} \sum_{t= 1}^T \sum_{s=1}^T \kappa \left(\frac{t-s}{bT}\right) \hat{u}_t \hat{u}_s'$$
with $\hat{u}_t = -R \left(G_T(\hat{\theta}_T)' W_T G_T(\hat{\theta}_T) \right)^{-1}W_T f(v_t;\hat{\theta}_T)$. In contrast to the traditional models with independent errors, the GMM framework does not require complete parameter specification, homoscedasticity, or independence.

\section{Estimators} \label{sec:estimators}

When the moment conditions are $f(v_t; \theta) = x_t (y_t- x_t'\theta)$ with weighting matrices $W_T = \mathbf{1}_p$ or $W_T = \Omega_w$, the GMM estimators are equivalent to ordinary least squares (OLS) and generalized least square (GLS) estimators, respectively \citep{hamilton1994time}. The OLS and GLS estimators can be expressed as
\begin{align}
    \hat{\theta}_{OLS} = & \underset{\theta \in \Theta}{\text{argmin}} \left (Y-X\theta \right)'\left (Y-X\theta\right) = (X'X)^{-1}X'Y \text{ and} \label{eq:ols_theta}\\
     \hat{\theta}_{GLS} = & \underset{\theta \in \Theta}{\text{argmin}} \left (Y-X\theta \right)'\Omega_w^{-1}\left (Y-X\theta \right) = (X'\Omega_w^{-1} X)^{-1} X'\Omega_w^{-1} Y. \nonumber
\end{align}

Then the number of moment conditions is the same as the number of parameters, i.e. $m = p$. Let $l_T$ be a $T \times 1$ vector of ones, and $\mathbf{1}_p$ be the identity matrix. If we assume a Gaussian location model, $y_t = \theta + w_t$, then the OLS and GLS estimators then simplify to 
\begin{align*}
    \hat{\theta}_{OLS} = &  l_T' Y /T = \overline{y} \text{ and}\\
    \hat{\theta}_{GLS} =&  \left [(l_T \otimes \mathbf{1}_p)' \Omega_w^{-1} (l_T \otimes \mathbf{1}_p) \right ]^{-1} (l_T \otimes \mathbf{1}_p)' \Omega_w^{-1} Y.
\end{align*}
The corresponding test statistics are
\begin{align}
     F_{OLS} = & \left [ \sqrt{T} R \left(\hat{\theta}_{OLS} - \theta_0 \right) \right ]' \hat{\Omega}_T^{-1}\left [ \sqrt{T}R \left(\hat{\theta}_{OLS} - \theta_0 \right) \right ]/d   \text{ and}  \label{eq:F_stat_ols}\\
      F_{GLS} = & \left [ \sqrt{T} R \left(\hat{\theta}_{GLS} - \theta_0 \right) \right ]' \hat{\Omega}_T^{-1}\left [ \sqrt{T}R \left(\hat{\theta}_{GLS} - \theta_0 \right) \right ]/d.\label{eq:F_stat_gls}
\end{align}

For simplicity, we use the Gaussian location model with $m=p$ and the hypotheses outlined in \eqref{eq:hypotheses_linear} unless otherwise stated. Observe, the test statistics \eqref{eq:wald_limiting}, \eqref{eq:F_stat_ols}, and \eqref{eq:F_stat_gls} rely on the estimators $\hat{\theta}_T$ and $\hat{\Omega}_T(\theta^*)$, where $\theta^*$ is used to construct the error terms. Note $\theta^*$ does not need to equal $\hat{\theta}_T$ for \eqref{lemma:wald_limiting} to hold; furthermore, it is typically more convenient both analytically and practically to use $\hat{\Omega}_T = \hat{\Omega}_T(\hat{\theta}_{OLS})$ \citep{hamilton1994time}, which we assume unless otherwise stated. Thus, using a Gaussian location model the moment conditions are $f(v_t; \theta) = y_t- \theta$, the estimated error terms are $\hat{w}_t = y_t - \hat{\theta}_{OLS}$,  we estimate $u_t$ using $\hat{u}_t = R\hat{w}_t$, and the autocovariance function is estimated via 
\begin{equation}
	\hat{\Gamma}(h) = \frac{1}{T} \sum_{t=1}^{T - h} \hat{u}_t\hat{u}_{t+h}'.  \label{eq:auto_cov_est}
\end{equation}

\subsection{Lugsail Kernels}
Lugsail estimators, proposed by \cite{vats2022lugsail}, can be used to offset the negative bias inherent to the kernels in $\mathscr{K}_1$. This class of estimators has two characterizations. It can be thought of as a linear combination of SV estimators with different bandwidths where $\kappa^* \in \mathscr{K}_1$,
\begin{equation}
   \hat{\Omega}^{(L)}_T = \frac{1}{1-c} \hat{\Omega}_{T, b} - \frac{c}{1-c}\hat{\Omega}_{T, b/r},
    \label{eq:Lugsail}
\end{equation}
or as a SV estimator with a transformed kernel function that belongs in a broader class of kernels. We refer to the first interpretation as the \textit{lugsail estimator}. For the second interpretation we refer to the transformed kernel as the \textit{lugsail kernel} and the original kernel function as the \textit{mother kernel}, which is part of class $\mathscr{K}_1$. This terminology is similar to that of steep origin kernels \citep{phillips2006spectral}. Let $\mathscr{K}_2$ be the set of lugsail kernels, i.e.\
\begin{multline*}
    \mathscr{K}_2 =  \SetLeft \kappa(\cdot): \mathbb{R} \rightarrow (-\infty, \infty) \SetLine \kappa^* \in \mathscr{K}_1,  c \in [0, 1), r\geq 1, \\
     \kappa(x) = \frac{1}{1-c} \kappa^*(x) - \frac{c}{1-c}\kappa^*(xr) \SetRight.
\end{multline*}


When $r=1$ or $c=0$, the lugsail kernel is equal to its mother kernel, i.e.\ $\mathscr{K}_1 \subset \mathscr{K}_2$.  Using a Bartlett mother kernel and setting $r = 1/c$ gives the flat-top Bartlett kernel \citep{politis1995bias, politis2011higher} while a standard Jackknife correction for a kernel $\kappa(\cdot)$ often takes the form $2 \kappa(x) - \kappa(x/\sqrt{2})$, which is a specific case of the lugsail with $c=1/2$ and $r = \sqrt{2}$ \citep{hardle2003note, hardle1990applied}.  In contrast to flat-top and Jackknife estimators, the lugsail framework enables finer control of the bias-variance tradeoff.

\begin{figure}[tbh]
    \centering
    \includegraphics[scale = 0.2]{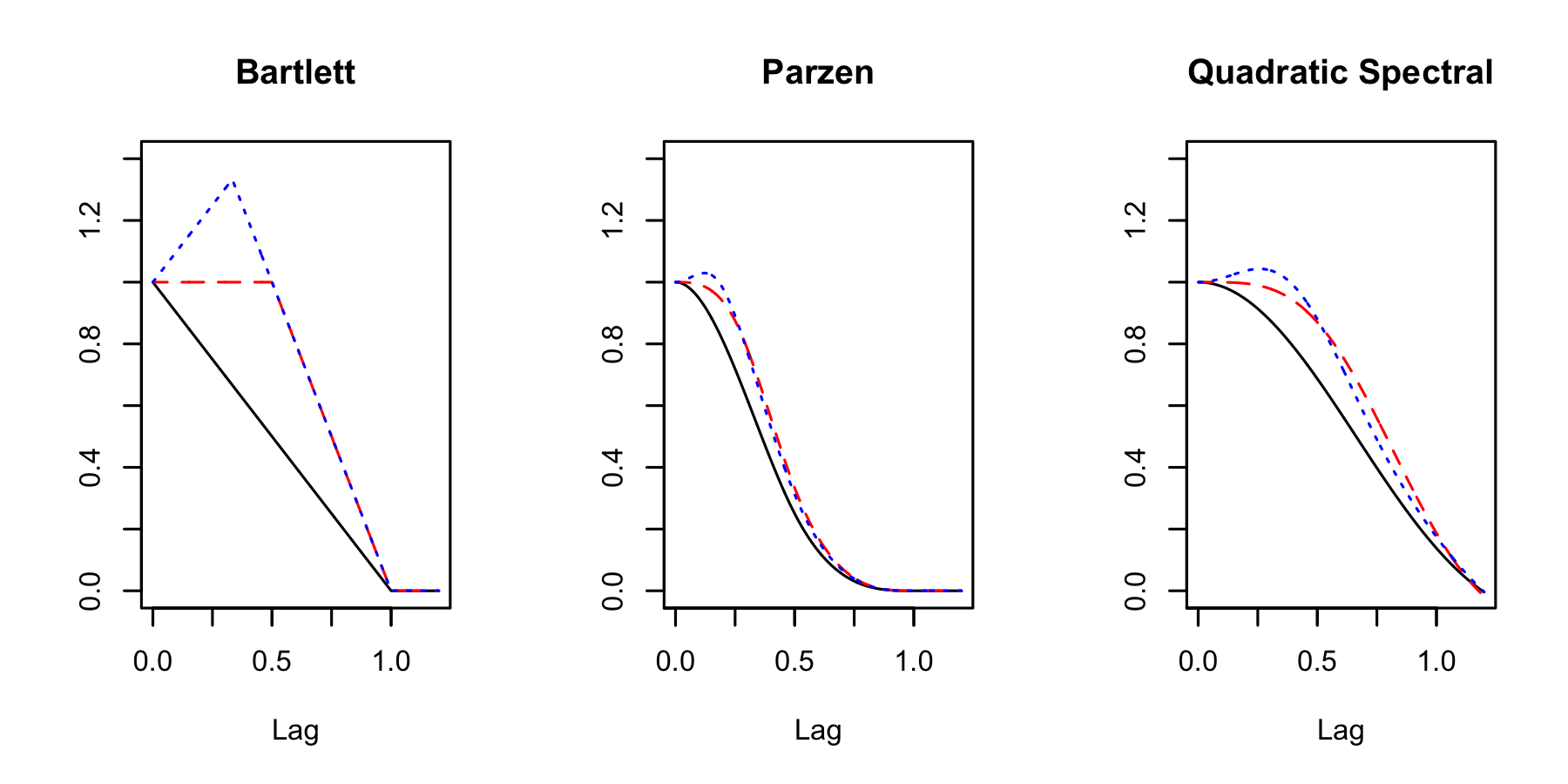}
    \caption[Lugsail Kernels, class $\mathscr{K}_2$]{Kernel functions in $\mathscr{K}_2$ with mother (solid black), zero (dashed red),  and over (dot-dashed blue) lugsail settings.}
    \label{fig:lugsail_kernels}
\end{figure}

Lugsail kernels do not necessarily share the same properties as the mother kernels. For example, zero lugsail kernel is designed to have an asymptotic kernel bias of zero, which will be apparent in Theorem \ref{thm:bias_var}.  Moreover, their weights may exceed 1 and no longer guaranteed to be non-negative. As a consequence, lugsail kernels do not guarantee psd estimates which we can verify by observing that there exists some $\theta \in [0, 2\pi]$ such that $K(\theta) \leq 0$.  For example, the Fourier transformation for a Bartlett lugsail kernel is
\begin{align*}
    K(\theta) = & \frac{1}{\pi (1-c) \theta^2 } \left \{ 1-\cos(\theta) -cr \left ( 1 - \cos \left (\frac{\theta}{r}\right ) \right ) \right \}.
\end{align*}

For simplicity, let $\kappa^* \in \mathscr{K}_1$ and $\kappa \in \mathscr{K}_2$ unless otherwise stated. For $\kappa \in \mathscr{K}_2$ we define the corresponding characteristic exponent as defined by \cite{parzen1957consistent}, 
\begin{equation}
    q = \max \SetLeft q: q \in \mathbb{N^+} \SetLine g_q = \left| \underset{x\rightarrow 0}{\lim} \frac{1-\kappa(x)}{|x|^{q}} \right | <\infty \SetRight 
    \label{eq:parzen_char}
\end{equation}
where $g_q$ is the $g^{th}$ generalized derivative of the kernel evaluated at the origin. The statistics $q$ and $g_q$ are useful for describing kernel characteristics. For the standard Bartlett, Parzen, and quadratic spectral kernels $q$ is equal to 1, 2, and 2, respectively. 

The aim of ($r, c$) is to inflate the weight function to offset the negative kernel and finite sampling bias inherent to the mother kernel.  \cite{vats2022lugsail} propose that if any arbitrary sequence of random variables that has an underlying process similar to that of an AR(1) model with correlation coefficient $\rho \in \{[0, 0.7), [0.7, 0.95), [0.95, 1)\}$ then we classify the situation as moderate, high, or extreme, respectively. The recommended settings for these three situations are in Table~\ref{tab:LugsailRecommendations}, which are illustrated in Figure~\ref{fig:lugsail_kernels} for three common mother kernels. 

\begin{table}[tbh]
    \centering
    \begin{tabular}{ccccc}
        \hline
      Direction &  Strength & Lugsail & $r$ & $c_r$\\
        \hline
        Negative& Any & Mother & 1 & 0 \\
       Positive & Moderate & Zero& 2& $r^{-q}$ \\
        Positive & Moderate-High &  Adaptive & 2& $\frac{\log(T) - \log(\left \lfloor bT \right \rfloor)+ 1}{r^{q}(\log(T) - \log(\left \lfloor bT \right \rfloor)) + 1}$\\
        Positive & High-Extreme & Over & 3& $\frac{2}{(1+r^q)}$\\
        \hline
    \end{tabular}
    \caption{Lugsail settings.}
    \label{tab:LugsailRecommendations}
\end{table}

\subsection{Centered Errors}

Let $\kappa_{bT}(t,s) = \kappa\left (\frac{t-s}{bT} \right )$, and observe $\kappa_{b}(\cdot,\cdot):[0,1]\times [0,1] \rightarrow \mathbb{R}$. We continue to use the SV estimator with lugsail kernels but instead of using the raw estimated errors we use the de-meaned estimated error, 
\begin{equation}
    \frac{1}{T} \sum_{t=1}^T \sum_{s=1}^T \kappa_{bT}\left(t,s \right ) (\hat{u}_t - \hat{\overline{u}})(\hat{u}_s -\hat{\overline{u}})'.
    \label{eq:LRV_centered_error}
\end{equation}
The estimator at \eqref{eq:LRV_centered_error} can be expressed as a transformation of the kernel function, i.e.\ 
\begin{equation}
    \hat{\Omega}_T=  \frac{1}{T} \sum_{t=1}^T \sum_{s=1}^T 
 \tilde{\kappa}\left( \frac{t-s}{bT }\right )\hat{u}_t \hat{u}'_s,
     \label{eq:LRV_centered_kernel}
\end{equation}
where  
\begin{equation}
    \tilde{\kappa}_{bT}(t,s) :=  \kappa_{bT}(t,s)  - \int^1_0 \kappa_b \left (\frac{t}{T}, r\right ) dr - \int^1_0 \kappa_b \left (\tau,\frac{s}{T}\right ) d\tau  + \int^1_0 \int^1_0 \kappa_b(r, \tau)  dr d\tau.
    \label{eq:centered_kernel}
\end{equation}
\noindent The kernel function $\tilde{\kappa}$ is known as the \textit{centered kernel}. 
The centered kernel function was investigated by \cite{velasco2001edgeworth, hall2000covariance,priestley1981spectral,lee2009testing,sun2014let,sun2012simple,zhang2013fixed}, and others. In GMM estimation procedures it is assumed that the model is correctly specified and that the moment conditions in \eqref{eq:moments} are valid. However, when tests are over specified or mis-specified the power suffers. A more powerful test is obtained when used the centered kernel regardless if the moment conditions in \eqref{eq:moments} are correct. Centered kernels also have an additional bias term due to centering the errors, which we refer to as the \textit{de-meaning bias}.

\subsection{Bias and Variance}

Recall $\Omega$ is comprised of an infinite sum of the autocovariance matrices and bandwidth parameter $b$ determines the proportion of autocovariance matrices that are given a non-zero weight.  Hence
\begin{equation*}
\Omega= \sum_{h = -\infty}^\infty \Gamma(h) = \sum_{h = -\left \lfloor bT\right \rfloor}^{\left \lfloor bT\right \rfloor} \Gamma(h) + \Omega_{FSB} \; , \label{omega_split}
\end{equation*}
where $\Omega_{FSB}$ is a function of $\left( b, T, \Omega \right)$ defined as
\begin{equation*}
\Omega_{FSB} = \sum_{h = -\infty}^{- \left \lfloor bT\right \rfloor - 1} \Gamma(h) + \sum_{h = \left \lfloor bT\right \rfloor + 1}^{\infty} \Gamma(h) .
\end{equation*}
The quantity $\Omega_{FSB}$, which we refer to as the \textit{finite sampling and bandwidth (FSB) bias}, reflects two sources of bias.  The first arises from finite sampling since only the first $T-1$ autocovariance matrices can be estimated in practice.  The second is due to the choice of bandwidth, which determines a truncation point after which the autocovariance matrices are no longer included in the estimation of $\Omega$.

The FSB bias is not the only source of bias when estimating $\Omega$.  The lag-$h$ autocovariance estimator in \eqref{eq:auto_cov_est} is biased by having $T$ in the denominator instead of $T - |h|$, which is used in most computer programs because it has less variability and is needed to ensure the estimate is psd when using $\kappa^* \in \mathscr{K}_1$. There is also bias attributed to the kernel function, which \cite{den199712, priestley1981spectral,parzen1957consistent} argue is typically the largest source of bias when using $\kappa^* \in \mathscr{K}_1$. We refer to this as the \textit{kernel bias}, which is usually negative, as the expected value of the sum of the difference between the down weighted autocovariances and their respective true values. We further distinguish the kernel bias from the de-meaning bias for convenience and transparency later on. 

Not explicitly represented here is the fact that these biases become more impactful when the number of restrictions under null hypotheses is larger. In addition, the strength of the correlation plays a large role in the estimation process. If we calculate the LRV for two data sets that only differ in correlation strength but use the same $b$ and $\kappa^* \in \mathscr{K}_1$, the data set with higher correlation will have a larger bias because of the down weighting mechanism.

Furthermore, these mechanics can influence each other. The weights generated by the kernel function are compounded by using a biased autocovariance estimator, bandwidth selection largely impacts kernel bias, and the sample size impacts bandwidth selection. This suggests that estimation mechanics should not be determined in isolation.  We must also issue a note of caution.  We have defined sources of bias through the estimation mechanics.  However, bias and variance are intimately related.  Each estimation mechanic above also effects the variability of $\hat{\Omega}_T$ which has important implications for the Type 1 and Type 2 error rates.

Individually, the bias components of the SV estimation process are not typically of concern and several are unavoidable. However, together these components can have a cascading effect on the test statistic and consequently the testing procedures. Most of these aspects cause the test statistic to be larger than otherwise expected, which leads to inflated Type 1 error rates for positively correlated second order stationary processes. 

We next expand the expression for the bias of $\hat{\Omega}_T$ in Theorem \ref{thm:bias_var} to incorporate some of the higher order bias terms described above.
\begin{theorem}[Bias and Variance]    \label{thm:bias_var}
Let $\kappa \in \mathscr{K}_2$, $q_0 \leq q$, where $q$ is the characteristic exponent of $\kappa$, and $g_q^*$ be the generalized derivative of the corresponding mother kernel. Let $c_1$ and $c_2$ be the first and second moment of $\kappa$, as defined in section \ref{sec:intro}. Let $\hat{\Omega}_T$ be constructed with $\tilde{\kappa}$, the corresponding centered kernel of $\kappa$. Assume that $b \rightarrow 0$ as $T \rightarrow \infty$. 
\begin{enumerate}
    \item[(a)] Under Assumption 4 and 5, 
    \begin{equation*}
            Bias(\hat{\Omega}_T) 
            =  \left ( \frac{1-cr^{q_0}}{1-c} \right ) (bT)^{-{q_0}} g^*_{q_0} h_{q_0} - \Omega c_1 b -\Omega_{FSB} b + o\left ((bT)^{-{q_0}} \right). 
    \end{equation*}
    \item[(b)] Under Assumption 5, 
    \begin{align*}
        Var(vec(\hat{\Omega}_T)) = & \left (c_2 b\right ) (I_{d^2} + \mathbb{K}_{dd}) \Omega\otimes \Omega + o(b) 
    \end{align*}
\end{enumerate}
where $\otimes$ is the Kronecker product, and $\mathbb{K}_{nm}$ is a $n \times m$ commutation matrix.
\end{theorem}

\begin{remark}
    For the first term in Theorem~\ref{thm:bias_var}a see \cite{parzen1957consistent} or \cite{andrews1991heteroskedasticity}. For the second term see \cite{priestley1981spectral} page 459,  and \cite{velasco2001edgeworth} Lemmas 2 and 6.  
\end{remark}

\begin{remark}
    For Theorem~\ref{thm:bias_var}b see \cite{ priestley1981spectral} or \cite{vats2022lugsail}. 
\end{remark}

Note all of the terms in Theorem~\ref{thm:bias_var} are asymptotic approximations. Through Theorem~\ref{thm:bias_var} we observe that the asymptotic kernel bias, de-meaning bias, and asymptotic variability of the SV estimator is reflected through $g_q$, $c_1$, and $c_2$, respectively. We also note that without the addition of $\Omega_{FSB}$ the expression for bias would be limited because it would only capture the effect of the kernel and transformations thereof. 

Table~\ref{tab:kernelstats} contains a summary of useful kernel statistics. Notice how the generalized derivative changes under the different settings. The mother kernel's generalized derivative is $g_q^*$, and the corresponding lugsail generalized derivative is $g_q = \left ( \frac{1-cr^q}{1-c} \right )g_q^*$. This indicates that the lugsail generalized derivative is typically less than the mother kernel's generalized derivative. We can also see the over lugsail is aptly named as it generates  an over correction, changing the sign for the expression of kernel bias. We observe from Table~\ref{tab:kernelstats} and Theorem~\ref{thm:bias_var} that the more inflated the kernel function is, the more $c_2$ increases, and hence the variability of $\hat{\Omega}_T$.

\begin{table}[hbt!]
    \centering
    \begin{tabular}{ll|ccc|ccc|ccc} \hline 
        & & \multicolumn{3}{c|}{\emph{Mother }}& \multicolumn{3}{c|}{\emph{Zero}}& \multicolumn{3}{c}{\emph{Over }}\\
    	& $q$& $c_1^*$& $c_2^*$& $g_q^*$ & $c_1$& $c_2$& $g_q$& $c_1$& $c_2$& $g_q$\\
    	\hline 
    	Bartlett & 1& 1.00& 0.67& 1.00 & 1.50 & 1.33& 0 & 2.00 & 2.33& -1.00\\
    	Parzen & 2 & 0.75 & 0.54 & 6.00 & 0.80& 0.060 & 0 & 0.81 & 0.57 & -6.00 \\
    	QS & 2&  1.25 & 1.00 & 1.42 & 1.52 & 1.29 & 0 & 1.51 & 1.31 & -1.42 \\
        \hline
    \end{tabular}
    \caption[Kernel Summary Statistics]{Kernel Summary Statistics.}
    \label{tab:kernelstats}
\end{table}

With the exception of the zero lugsail kernel, lugsail kernels will have the same characteristic component ($q$) as the mother kernel. Due to its `flatness' the zero-lugsail kernel has an infinite characteristic component, similar to flat top and truncated kernels ($\kappa(x) = 1$ for $x<b$, 0 elsewhere).  In these cases the term $q_0$ in Theorem \ref{thm:bias_var} is limited by the underlying model instead of the kernel function, which was observed by \cite{politis1995bias} for flat top kernels. For the standard class of autoregressive processes with moving average (ARMA) residuals we can think of the spectral density as possessing an essentially infinite number of derivatives. This is a major advantage for the zero lugsail kernel function because as seen by Theorem \ref{thm:bias_var} the bias of $\hat{\Omega}_T$ is dominated by $(bT )^{-q}$. 

There is still the issue of approximating the FSB bias term $\Omega_{FSB}$.  We suggest using the simple robust strategy in Remark~\ref{remark:fsb_bias} and note similar techniques are common \citep{andrews1991heteroskedasticity, liu2021batch, muller2014hac}.

\begin{remark}[FSB Bias] \label{remark:fsb_bias}
An expression for $\Omega_{FSB}$ can be approximated using classical time series models. Let $\kappa \in \mathscr{K}_2$, and  assume that $b \rightarrow 0$ as $T \rightarrow \infty$. Then there exists some value $ c_{b,T}$ as a function of $(b, T)$ such that the asymptotic FSB bias is $\Omega_{FSB} = \Omega c_{b,T}$. For example, if the underlying process is similar to an AR(1) model with correlation coefficient $\rho$ then value $c_{b, T}$ may be approximated by 
\begin{equation*}
c_{b,T} = \frac{2\hat{\rho}^2}{1+\hat{\rho}}\hat{\rho}^{bT}.
\end{equation*}
\end{remark}

\section{Limiting Behavior} \label{sec:limiting}

\subsection{Fixed-Smoothing Distribution}

Inference procedures usually rely on familiar limiting distributions, instead of exact distributions, because they are more accessible, highly versatile, and fairly accurate provided the sample size is large. For example, a standard univariate one sample mean test using independent and identically distributed normal data can utilize the normal distribution or the t-distribution, the former being the limiting distribution and the latter being the exact distribution.  

Two limiting distributions are popular for inference procedures in this space. The first limiting distribution is under the assumption that $b\rightarrow0$ as $T \rightarrow \infty$. This set of assumptions and related results are referred to as \textit{adjusted-smoothing asymptotics}, which leads to $d F_T \overset{\cdot}{\sim }\chi^2_d$.  

In contrast, \cite{ kiefer2005new} hold $b$ as a fixed value as $T \rightarrow \infty$. This set of assumptions and related results are referred to as \textit{fixed-smoothing asymptotics} where the limiting distribution is non-standard and lacks a closed form expression. It is also known as fixed-$b$ \citep{sun2014let, hualde2015small} or fixed-$K$ \citep{sun2013heteroskedasticity} asymptotics. 
The fixed-smoothing limiting distribution changes in response to the bandwidth and kernel selected as opposed the adjusted-smoothing limiting distribution which remains constant despite the known impacts on the behavior of the test statistic.

The following proposition obtains expressions for the distribution of a fixed-smoothing random variable $F_\infty(d,b) \sim \mathscr{F}_\infty(d, b)$ applicable for kernels $\kappa \in \mathscr{K}_2$ and their corresponding center kernel. 
\begin{proposition} [Fixed-Smoothing Distribution]
Let $\kappa \in \mathscr{K}_2$ and it's corresponding center kernel be $\tilde{\kappa}$. Let $b \in [0,1]$ be a fixed value, and $T \rightarrow \infty$. 
\begin{enumerate}
    \item[(a)] Under assumptions~\ref{assumption:1}-\ref{assumption:4}, observe that 
    \begin{align}
    F_T \overset{d}{\rightarrow}&  \mathscr{F}_{\infty}(d, b) \nonumber\\
    =& B_d(1)' Q_d(b)^{-1} B_d(1)/d \nonumber\\
    =& B_d(1)' \left [\int^1_0 \int^1_0 \kappa\left(\frac{r-s}{bT}\right) d\Tilde{B}_d(r) d \Tilde{B}_d(s) \right ]^{-1}B_d(1)/d \nonumber
    \end{align}
    \noindent where $Q_d(b)$ is a random variable independent of the standard $d$-dimensional Brownian process $B_d(1)$, and $\Tilde{B}_d(r)$ is standard $d$-dimensional Brownian bridge process.
    \item[(b)] The following are all equal,
    \begin{align}
        \int^1_0 \int^1_0 \kappa\left ( \frac{t-s}{b}\right ) d\tilde{B}_d(t) d \tilde{B}_d(s) = & \int^1_0\int^1_0  \tilde{\kappa}\left ( \frac{t-s}{b}\right ) d\tilde{B}_d(t) d \tilde{B}_d(s)\nonumber\\
        =&  \int^1_0\int^1_0  \tilde{\kappa}\left ( \frac{t-s}{b}\right ) dB_d(t) d B_d(s) \label{eq:Q_no_bridge}. 
    \end{align}
    \item[(c)] For fixed-smoothing random variable $F_\infty(d,b)$ with kernel $\tilde{\kappa}$, 
    \begin{equation*}
    dF_{\infty}(d,b) \overset{d}{=} \frac{||\eta||^2}{v_{11\cdot 2}}
    \end{equation*}
    \noindent where $\eta \sim N(0, \mathbb{I}_d)$, and $v_{11 \cdot 2} =  v_{11} - v_{12}v_{22}^{-1}v_{21}$  is a random quantity such that
\begin{equation*}
    \begin{bmatrix}
    v_{11} & v_{12} \\ 
    v_{21} & v_{22} 
    \end{bmatrix}
\end{equation*}
\noindent is the sum of independent but not necessarily identically distributed Wishart random variables where $v_{11} \in \mathbb{R}$, $v_{21} \in \mathbb{R}^{(d-1)\times 1}$ , $v_{12}\in \mathbb{R}^{1 \times (d-1)}$, and $v_{22} \in \mathbb{R}^{(d-1) \times (d-1)}$. Furthermore,  $E[v_{11\cdot 2}] =1 - b c_1 - b c_2 (d-1) + o(b)$ and $ E[v_{11\cdot 2}^2]= 1 - 2b(c_1 - c_2) -2(d-1) b c_2 + o(b)$.
\end{enumerate}
\label{prop:fixed_b_forms}
\end{proposition}

Proposition \ref{prop:fixed_b_forms}a is the classical expression of the fixed-smoothing distribution which showcases that $\mathscr{F}_\infty(d, b)$ can be decomposed into two independent components with $Q_d(b)$ in the denominator and a Brownian process in the numerator.   The expression in Proposition \ref{prop:fixed_b_forms}b provides an alternative representation for the function $Q_d(b)$ when using the corresponding centered kernel $\tilde{\kappa}$.  The expression in Proposition \ref{prop:fixed_b_forms}b allows us to rewrite $Q_d(b)$ in terms of Brownian motions instead of Brownian bridge processes. 

Proposition \ref{prop:fixed_b_forms}c utilizes the representation in \eqref{eq:Q_no_bridge} to construct an alternative expression for the $\mathscr{F}_\infty(d, b)$ distribution when utilizing centered errors and a lugsail kernel.  This representation mirrors equation (10) in \cite{sun2014let} by using a Fourier series expansion of \eqref{eq:Q_no_bridge} when relying on $\kappa \in \mathscr{K}_2$ to construct a center kernel $\tilde{\kappa}$. 

Proposition~\ref{prop:fixed_b_forms} gives different forms for expressing a fix-smoothing random variable which helps us derive an expression for the distribution function in Theorem~\ref{thm:null_alt_distr}.

\begin{theorem} 
 Let $\kappa \in \mathscr{K}_2$ and it's corresponding center kernel be $\tilde{\kappa}$. Assume that $b \rightarrow 0$ as $T \rightarrow \infty$. Denote $G_{d,\delta^2}(z)$ as the cumulative distribution function function of a $\chi^2$ random variable with degrees of freedom $d$ at $z$, and non-centrality parameter $\delta^2$. Let $w_q = d^{-1}trace\left(\sum_{j = -\infty}^\infty |j|^q \Gamma(j) \Omega^{-1}\right)= d^{-1}trace\left(h_q \Omega^{-1}\right)$, and $G_{d}(z) = G_{d,0}(z)$ for simplicity. 
 \begin{enumerate}
    \item[(a)] For fixed-smoothing random variable,
    $$P(d F _{\infty}(d, b) \leq z) = G_d(z) + G_d''(z)z^2 c_2 b- G_d'(z) z\left [c_1 + c_2 (d-1)\right]b  + o(b).$$
    \item[(b)] Using Remark~\ref{remark:fsb_bias}, under Assumptions 4 and 5,
    \begin{align*}
        P_\delta(dF_{OLS} \leq  z)  =&   G_{d, \delta^2}\left (z\right) + G^{''}_{d, \delta^2}\left (z\right) z ^2 c_2  b- G'_{d, \delta^2}(z) z[c_1 + c_2 (d-1)]b  \\&  -  G'_{d, \delta^2}(z) zc_{b,T} - (bT)^{-q}G'_{d, \delta^2}\left (z\right)zg_q w_q + o(b) + o \left( (bT)^{-q}\right).
    \end{align*}
 \end{enumerate}
\label{thm:null_alt_distr}
\end{theorem}

The results in Theorem \ref{thm:null_alt_distr} are similar to that of \cite{sun2011robust, sun2014let} and \cite{lazarus2018har, lazarus2021size}.  The chief differences are the inclusion of the lugsail kernels and incorporation of a FSB bias term $c_{b, T}$. Both expressions are in terms of the $\chi^2$ distribution, where beyond the first-order term we see the discrepancy between the distribution for the given random variable and standard $\chi^2$ distribution that would be used under the standard adjusted-smoothing asymptotic theory.  

Theorem \ref{thm:null_alt_distr}a provides a general expression for the null distribution of a fixed-smoothing random variable $F_\infty(d, b)$, which helps understand the different components that will be captured. 
In contrast, in Theorem \ref{thm:null_alt_distr}b provides a general expression the test statistic in a finite sampling setting, and is where we observe a difference from prior results.  

From this theorem we observe that the adjusted-smoothing $\chi^2$ limiting distribution overlooks crucial components. Notably, the kernel and FSB bias terms are not prominent due to the assumption of $T \rightarrow \infty$, causing these terms to converge to zero. However, the fixed-smoothing distribution still captures significant components such as variability, de-meaning bias, and dimensionality bias, represented by coefficients $c_2$, $c_1$, and $c_2(d-1)$, respectively.

\begin{remark}
There is the practical consideration for estimating $w_q$ which is dependent on the kernel through its characteristic exponent. Rudimentary but effective estimates are available by relying on an AR(1) model with auto regressive coefficient $\rho$ \citep{andrews1991heteroskedasticity, lazarus2021size, chang2018understanding}.  This gives us $w_1 = \frac{2\rho}{1 - \rho^2}$, and $w_2 =\frac{2\rho}{(1-\rho)^2} $.  Moreover, this approach we can easily obtain approximations for $(bT)^{-q} g_q w_q$ for all kernels $\kappa \in \mathscr{K}_2$. 
\end{remark}

\subsection{Rejection Rates}

With Theorem~\ref{thm:null_alt_distr} and Remark~\ref{remark:fsb_bias}, we can derive accessible expressions for Type 1 and Type 2 error when utilizing fixed-smoothing critical values that incorporate the FSB bias. 

\begin{corollary}
Let $\kappa \in \mathscr{K}_2$, it's corresponding center kernel be $\tilde{\kappa}$, and $\chi^\alpha_d$ be the $1-\alpha$ quantile from the $\chi^2_d$ distribution. Assume that $b \rightarrow 0$ as $T \rightarrow \infty$. Using test statistic $d F_{OLS}$ and fixed-smoothing critical values, under Assumptions 4 and 5,
    \begin{enumerate}
        \item[(a)] $e_I(b) \approx  \alpha + G'_d(\chi^\alpha_d)\chi^\alpha_d c_{b, T} + (bT)^{-q}G'_d(\chi^\alpha_d)\chi^\alpha_dg_q w_q$
        \item[(b)] $1 -e_{II}(b) \approx 1 - G_{d, \delta^2} (\chi^\alpha_d) - \left [ \frac{\delta^2}{2} G^{'}_{(d+2), \delta^2}(\chi^\alpha_d) - G'_{d, \delta^2}(\chi^\alpha_d)(d-1)\right ] \chi^\alpha_d c_2 b  \\+ G'_{d, \delta^2}(\chi^\alpha_d) \chi^\alpha_d c_{b,T} + G'_{d, \delta^2}(\chi^\alpha_d) \chi^\alpha_d c_1b + (bT)^{-q}G'_{d, \delta^2}   (\chi^\alpha_d)\chi^\alpha_d g_q w_q$.
    \end{enumerate}
    \label{corollary:error}
\end{corollary} 

The expression in Corollary~\ref{corollary:error}a helps illustrate what a fixed-smoothing critical value captures for a finite data set. That is, the Type 1 error rate and the distortion thereof. The distortion arises from the kernel bias and the FSB bias, the only bias terms in Theorem \ref{thm:bias_var} that explicitly disappear as the sample size goes to infinity. 

In Corollary \ref{corollary:error}b we obtain another representation for Theorem \ref{thm:null_alt_distr}b with additional terms to better approximate the behavior of $F_T$ under the alternative hypothesis. Quantifying Type 2 error (and power) in a robust inference setting is not as straight forward as size distortion. In general, overall power will decrease as variability increases since distributions under the null and alternative are difficult to distinguish if the tails are heavy. In contrast, the negative bias of the LRV has the opposite effect on power as it does for Type 1 error in that a negative bias improves power. 

It has been well established that the standard SV estimator with $\kappa^* \in \mathscr{K}_1$ is negatively biased. We further note that the SV estimator $\hat{\Omega}_T$ has a positively skewed distribution (see Theorem \ref{thm:null_alt_distr} for connections to $\chi^2$ distribution). This asymmetric density indicates the majority of estimates will be smaller than $E(\hat{\Omega}_T)$, which is already less than the true value $\Omega$. These issues cause an inflated Type 1 error rate, which is what the fixed-smoothing theory and lugsail estimator address. In Corollary \ref{corollary:error}a we see the higher level terms from Theorem 3a disappear when utilizing fixed-smoothing critical values, but terms intrinsic to bias persist. The zero lugsail estimator has a generalized derivative equal zero, resulting in a zero asymptotic kernel bias which further decreases the size distortion in comparison to the classical mother kernels. 



We rely on our choices for the bandwidth and kernel function to control testing behavior. As the bandwidth increases, the Type 1 error rate converges to $\alpha$, while the Type 2 error rate follows a concave function. Optimizing solely for Type 1 error results in a bandwidth of 1, which is impractical. Similarly, minimizing the Type 2 error results in bandwidths of either 0 or 1, depending on the context, which is also impractical. Therefore, controlling the testing mechanics effectively requires striking a balance.

\section{Bandwidth Rules} \label{sec:bandwidth}

Several works have focused on optimizing SV estimators via bandwidth rules for kernels in $\mathscr{K}_1$ \citep{andrews1991heteroskedasticity, lazarus2021size, priestley1981spectral}. The foremost rule for bandwidth selection minimizes asymptotic mean squared error was proposed by \cite{andrews1991heteroskedasticity}.  For the Bartlett mother kernel with VAR(1) random errors and correlation coefficient matrix $\rho \mathbf{1}_d$, this rule simplifies to
\begin{equation}
    b_{mse} =  1.1447 \left ( \frac{2 \hat{\rho}}{(1- \hat{\rho})^2} \frac{1}{T}\right) ^{2/3}.
    \label{eq:mse_rule_simple}
\end{equation}
Although \eqref{eq:mse_rule_simple} is straightforward to apply, it solely addresses correlation concerning the asymptotic kernel bias and is unsuitable for kernels where the predominant asymptotic bias term is zero. In such cases, optimization based solely on asymptotic variance leads to $b_{mse}=0$, yielding a biased estimate that is clearly non-optimal.

Recent work has developed loss functions and bandwidth rules explicitly for inference aiming to balance Type 1 and Type 2 errors \citep{lazarus2021size, lazarus2018har, sun2011robust, sun2013heteroskedasticity,sun2008optimal}.  Since inference procedures aim to have an observed Type 1 error rate to be no greater than $\alpha$, most these rules prioritize size control, with power incorporated via a tuning parameter $\tau \in [0, 1]$.  To date the choice of $\tau$ is largely ad hoc, with general guidance emphasizing Type 1 error.  

The loss function we select aims to prioritize generating a valid Neyman-Pearson like testing procedures, which is similar in spirit to the loss function of \cite{sun2014let}. However, we account for FSB bias and use an adaptive $\tau$ that selects a $b$ where Type 1 error rate begins to stabilize, i.e.\ when the improvement on Type 1 error from increasing $b$ becomes increasingly marginal. 

\subsection{Zero-Lugsail Bandwidth Rule}
\label{sec:zero_bandwdith_rule}

We choose the smallest bandwidth where the Type 1 error is in the neighborhood of $\alpha$. The optimization rule can be summarized as follows, 
\begin{align}
    b_{opt} = & \min \left \{b\in [0,1]: e_I(b) \leq \alpha + \tau \right \}.  \label{eq:bandwidth_rule}
\end{align}
For zero lugsail kernels with FSB bias approximated in Remark~\ref{remark:fsb_bias} this simplifies to
\begin{align}
   b_{opt} =  & \min \left \{b\in [0,1]: b \geq \log \left( \frac{\tau}{G_d'(\chi^\alpha_d) \chi^\alpha_d}\frac{1+ \rho}{ 2 \rho^2}\right) \frac{1}{\log(\rho)}\frac{1}{T} \right \}.
     \label{eq:bandwidth_rule_zero}
\end{align}
The curve of the Type 1 error rate in Corollary~\ref{corollary:error}a as a function of $b$ results in a shape of an `L' or an arm illustrated in Figure~\ref{fig:Type1Error_OptimalbZero}. This arm like shape is commonly observed in loss functions, and the general consensus is to select a tuning parameter that is located near the `elbow' (the bend of the curve), or just past it \citep[see e.g.][]{james2013introduction}. We observe that the curve converges to origin as the correlation decreases or sample size increases. It also shifts both horizontally and vertically as the significance level changes, generally to the origin as $\alpha$ decreases. 

\begin{figure}[htb]
    \centering
    \includegraphics[scale =.55]{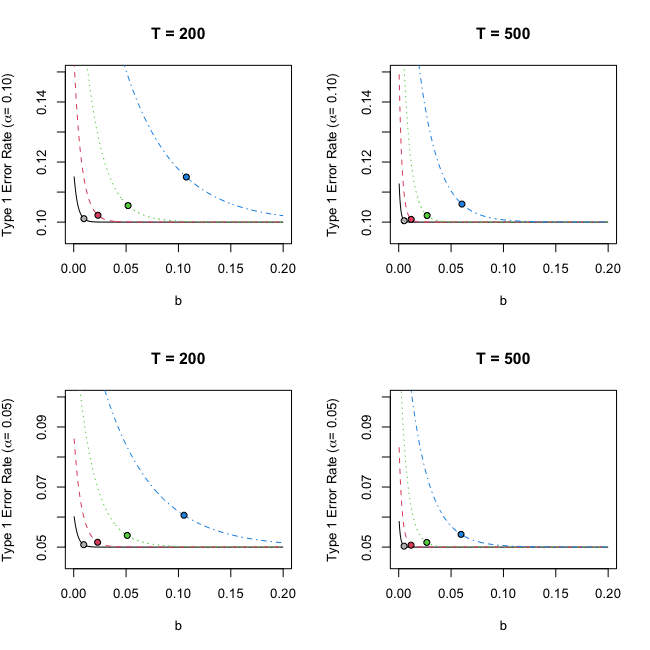}
    \caption[Expected Type 1 Error Rate for Zero Lugsail Kernels using Various Bandwidth Rules and Settings.]{Expected Type 1 error rate of zero lugsail kernels as a function of $b$ for $d=1$, $\alpha = (0.05, 0.01)$, $T = (200, 500)$, and $\rho$ = .25 (solid black), 0.50 (dashed red), 0.75 (dotted green), 0.90 (dot-dashed blue). Circles indicate Type 1 error rate at $b_{opt}$.}
    \label{fig:Type1Error_OptimalbZero}
\end{figure}

If we select an arbitrary constant $\tau \in \mathbb{R}^+$, we risk the possibility that the Type 1 error curve never falls in such a neighborhood. In contrast to other tolerance rules, which select ad-hoc neighborhoods, our rule relies on the mathematical properties of our Type 1 error rate expression and always results in a bandwidth between 0 and 1.  Specifically, we propose
\begin{align}
    \tau = & - \frac{\alpha^{1/(2d)}}{T \log(\rho)},\label{eq:tau_rule}
\end{align}
to select a bandwidth where the Type 1 error rate begins to stabilize, which is related to the derivative of the Type 1 error rate. Instead of finding a bandwidth that minimizes Type 1 error rate, i.e. where the derivative equals 0, we select a point where it is \textit{near} 0. 

The rule \eqref{eq:tau_rule} picks the smallest bandwidth $b$ at the point at which the Type 1 error rate becomes mostly flat instead of mostly vertical regardless of the correlation, sample size, or significance level. We get this effect for any $\tau_0 \in [0,1]$ such that $\tau =- \frac{\tau_0}{T \log(\rho)}$, with values closer to zero having a farther clearance from the elbow (i.e.\ where the slope of the tangent line to the curve is -1). 
Using $\tau_0 = \alpha^{1/(2d)}$ creates a small clearance from the elbow that naturally tightens for higher significance levels and accounts for dimensionality bias, which greatly effects the behavior of the testing procedure. 
For example, with $\alpha=0.05$ and $d= 1$, the bandwidth rule selected is at the point were a 1\% increase in $b$ is associated with a 0.22\% decrease in Type 1 error rate. Meaning we observe an increasingly marginal decrease in error compared to the increase in bandwidth. 


Figure~\ref{fig:Type1Error_OptimalbZero} shows the asymptotic Type 1 error rate under various sample sizes, significance levels, and correlations.  The bandwidth $b_{opt}$ at \eqref{eq:bandwidth_rule_zero} is indicated with a circle, and adjusts to the curve as expected.  Highly correlated data is expected to digress from the origin, and thus the $b_{opt}$ rule is larger. Similarly, smaller data sets also digress from the origin. 

%
Notice the optimal bandwidth decreases as $\alpha$ decreases despite enforcing a tighter neighborhood, which is because the curve almost always shifts horizontally to the origin as the significance level increases.

The bandwidth rule $b_{opt}$ in \eqref{eq:bandwidth_rule_zero} with $\tau$ in \eqref{eq:tau_rule} is specifically tailored for zero lugsail kernels.  Unlike many other bandwidth rules that rely on the kernel function's properties via $q$, this proposed rule is grounded on the characteristic of zero lugsail estimators, where $q = \infty$.

Flat top kernels, proposed by \cite{politis1995bias}, also have $q = \infty$ and are equivalent in some settings, e.g.\ the Bartlett kernel. While the proposed rule does not immediately apply to general flat top kernels, it's plausible to assume that with slight adjustments, the findings could be broadened to encompass them. 



\subsection{Extensions to the Zero-Lugsail Bandwidth Rule}

The bandwidth rule $b_{opt}$ in \eqref{eq:bandwidth_rule_zero} differs from most other rules because of its reliance on FSB bias instead of kernel bias. It is not immediately applicable to non-zero lugsail settings, i.e. mother, adaptive, or over. To use a comparable rule with adaptive or over lugsail settings we recommend substituting a candidate $b$ value (e.g. the rule at \eqref{eq:bandwidth_rule_zero}) for the kernel bias term in the Type 1 error expression in Corollary 1a and continue using the loss function in \eqref{eq:bandwidth_rule} with the recommended $\tau$ in \eqref{eq:tau_rule}. This results in a smaller version of \eqref{eq:bandwidth_rule_zero}, which is expected since over and adaptive lugsail kernels induce a positive bias.

For mother kernels, we cannot use the same procedure without modifying $\tau$ because the negative kernel bias causes the curve to shift vertically away from the origin and therefore there is no longer a $b$ that meets the criteria in \eqref{eq:bandwidth_rule}. Instead we recommend using the zero bandwidth rule at \eqref{eq:bandwidth_rule_zero} as a rudimentary and effective approach. This obtains an optimal bandwidth for mother kernels according to their finite sampling bias exclusively, instead of the kernel bias exclusively which contrasts what most other bandwidth rules consider. 


Throughout this document we have been under the premise that the data is positively correlated. With negatively correlated data the issue of bias is less problematic because of the oscillatory nature of the autocovariance function. Although one may substitute $\rho$ with $|\rho|$ in the bandwidth rule \eqref{eq:bandwidth_rule_zero} with a zero lugsail kernel, it would result in an unnecessarily large loss in power. The mother kernel with the same bandwidth rule would be a more reasonable approach. The major issue for inference with negatively correlated data is not bias but the distributional properties of $\hat{\Omega}_T$. 

Lastly, the bandwidth rule we suggested is built with an estimated $\rho$. For multivariate data multiple values for $\rho$ can be considered, but we suggest using a weighted average or selecting the largest value estimated. Both techniques are common practice and have been used by \cite{andrews1991heteroskedasticity} and \cite{liu2021batch}. 

\subsection{Practical Considerations} \label{subsec:practical}


Estimates for $\hat{\Omega}_T$ need to be positive definite so the inverse exists and the test statistic $F_T$ is defined. The lack of positive definiteness is not a new issue for LRV estimators. The truncated kernel and flat top kernels also suffer from this problem, where several corrections have been purposed and used by \cite{mcmurry2010banded,mcmurry2015high, vats2022lugsail}. In particular, appendix A of \cite{mcelroy2022estimating} gives a brief review and comparison of adjustments, including clip correction, selective shrinkage to positive definiteness, shrinkage towards white noise, shrinkage towards a 2nd order estimate, and a thresholding correction. In addition, an eigenvalue correction proposed by \cite{jentsch2015covariance} can also be considered. 

For simplicity, we suggest a simple correction that generates a consistent estimator as $b \rightarrow 0$ as $T \rightarrow \infty$. Let $\hat{\Omega}^{(L)}_T$ be the unadjusted lugsail estimator, $\hat{\Omega}^{(M)}_T$ be an SV estimator generated with $\kappa^* \in \mathscr{K}_1$, and $I$ be the set of indices such that $\left[ \hat{\Omega}^{(L)}_T\right ]_{i,i} \leq 0$. Define the new psd corrected lugsail estimator as \begin{equation}
    \left[ \hat{\Omega}^+_T\right ]_{i,j} = \begin{cases}
\left[ \hat{\Omega}^{(M)}_T\right ]_{i,j} & \text{ if } (i=j) \text{ and } i\in I \\ 
\left[ \hat{\Omega}^{(L)}_T\right ]_{i,j} & \text{ if otherwise.}  
\end{cases}.
    \label{eq:lugsail_psd}
\end{equation}
In short, $\hat{\Omega}^+_T$ replaces non-positive diagonals using the diagonals generated by the mother a kernel. 

%


\begin{remark}
Because a psd estimate is no longer guaranteed for lugsail estimators, the motivation for using the biased autocovariance estimator $\hat{\Gamma}(s)$ diminishes. We can instead revert back to the unbiased autocovariance estimator $\frac{T}{T-s}\hat{\Gamma}(s)$ to further decrease bias, a similar observation was made by  \cite{politis1995bias, priestley1981spectral} for other non-psd kernels. 
\end{remark}

\section{Examples and Discussion} \label{sec:simulation}

To assess our findings we conduct a simulation study with common scenarios in time series and econometric applications. We concentrate on various model structures and compare Type 1 error rates. Results were generated using 1000 simulations, $\alpha = 0.05$,  simulated fixed-smoothing critical values, and the Bartlett kernel with mother and zero lugsail settings. 

\subsection{Accuracy of Approximation} \label{sec:simulation1}
First consider a Gaussian location model with VAR(1) error terms $w_t$ and independent standard normal disturbances $\epsilon_t$. We simulate results for a VAR(1) process with $T=500$, $d = (1, 4)$, $\rho = (0, 0.75)$, adaptive-smoothing and fixed-smoothing critical values, and the mother and zero lugsail settings. For each scenario we present plots of bandwidth versus Type 1 error rate, bandwidth versus Type 2 error rate, and noncentrality parameter versus power. For Type 2 error rate we must select a point under the alternative hypothesis, where it is customary to select a point $\delta^2$ under the alternative such that the asymptotic power is $0.50-0.75$ \citep{elliott1992efficient, sun2014let}. That is, by selecting $\delta^2$ such that $P_{\delta^2} \left ( \varrho  \geq \chi^\alpha_d\right) = 0.50 $ (or $0.75$) where $\varrho \sim \chi^2_{d, \delta^2}$. 

Our choice of an alternative hypothesis is motivated by the concept of size-adjusted power loss, which was conceived to give a fairer comparison across procedures \citep{lazarus2018har}.  Size-adjusted power loss is the difference in power when using asymptotic critical values and exact critical values and can also be thought of as the point at which the true power data is the least represented by the asymptotic power.  We generate data under the alternative hypothesis $H_A:  R (\theta -\theta_0)= R \Lambda T^{-1/2}\delta$ where $\delta^2$ corresponds to the point where asymptotic power is approximately $0.66$, since \cite{lazarus2018har} illustrate through simulation that the maximum size adjusted power loss ($\Delta_p^{max}$) for mother kernels typically occurs at this point regardless of $d$.   This motivation lead to $\delta^2$ values of 2.43 and 4.27 for $d$ equal to 1 and 4, respectively.


We further note that the AMSE rule $b_{mse}$ at \eqref{eq:mse_rule_simple} was founded for mother kernels using adaptive-smoothing critical values, the classical approach for robust estimation, and $b_{opt}$ at \eqref{eq:bandwidth_rule_zero} was founded for zero lugsail kernels with fixed-smoothing critical values.  We have indicated the $b_{mse}$ and $b_{opt}$ values for their respective scenarios in the Type 1 and Type 2 error plots.  

In addition, we plot the performance for mother kernels using fixed-smoothing critical values and zero lugsail kernels with adaptive-smoothing critical values to better observe the effects of the two. We have not indicated any specific bandwidth points under these scenarios for the error plots. However, the power function must be plotted using a specified bandwidth, we chose $b_{mse}$ for mother kernels and $b_{opt}$ for zero lugsail kernels for convenience. 

As a comparison we include the difference-based estimator from \cite{chan2022optimal} using the \texttt{dlvr} \texttt{R} package referenced therein.  Note there is no bandwidth parameter for this estimator, as such the Type 1 and Type 2 error curves are constant.

\begin{figure}[htb]
    \centering
    \includegraphics[scale=.16]{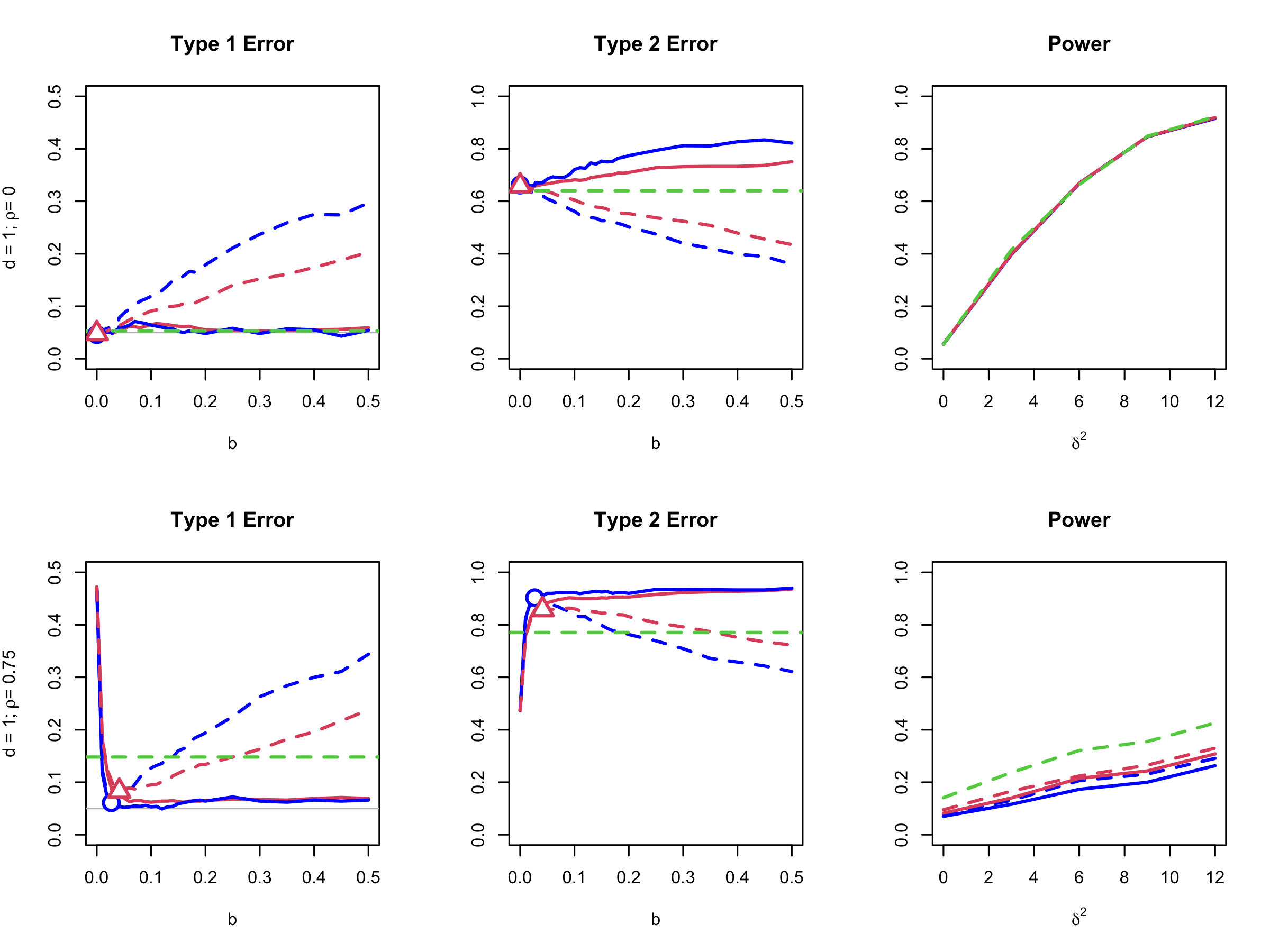}
    \caption[Inference Performance Metrics for a AR(1) process with $d = 1$.]{Inference performance metrics for a VAR(1) process with $d = 1$. Dashed lines indicate adaptive-smoothing critical values, and solid lines are fixed-smoothing critical values. Red indicates the mother Bartlett kernel, blue indicates the zero lugsail Bartlett kernel, and green indicates the difference-based estimator. The triangles and circles indicate $b_{mse}$ and $b_{opt}$ respectively. The power curve has been drawn using $b_{mse}$ for the mother kernels, and $b_{opt}$ for the zero lugsail kernels. The Type 2 error rate is at a fixed point under the alternative hypothesis.}
    \label{fig:performance1}
\end{figure}



\begin{figure}[htb]
    \centering
    \includegraphics[scale=.16]{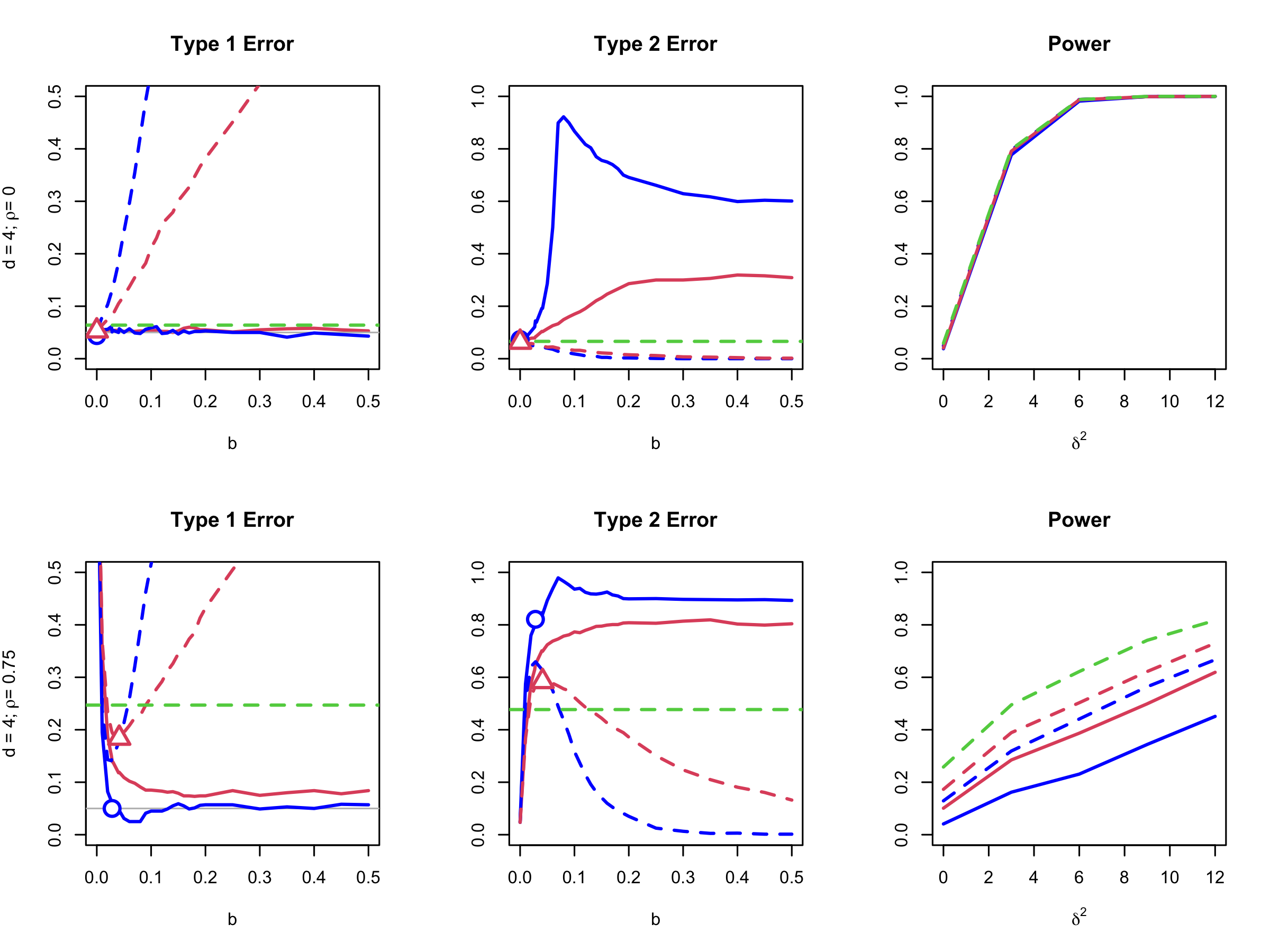}
    \caption[Inference Performance Metrics for a VAR(1) process with $d = 4$.]{Inference performance metrics for a VAR(1) process with $d = 4$. Dashed lines indicate adaptive-smoothing critical values, and solid lines are fixed-smoothing critical values. Red indicates the mother Bartlett kernel, blue indicates the zero lugsail Bartlett kernel, and green indicates the difference-based estimator.  The triangles and circles indicate $b_{mse}$ and $b_{opt}$ respectively. The power curve has been drawn using $b_{mse}$ for the mother kernels, and $b_{opt}$ for the zero lugsail kernels. The Type 2 error rate is at a fixed point under the alternative hypothesis.}
    \label{fig:performance4}
\end{figure}

For $\rho = 0$, the simulation results are presented in the top rows of Figures~\ref{fig:performance1} and~\ref{fig:performance4} for $d=1$ and $d=4$, respectively.  Here the optimal bandwidth is $b_{mse} = b_{opt} = 0$ and the LRV estimators are the same for the mother and zero lugsail settings. The discrepancy between the performance of the Type 1 error and Type 2 error plots is largely due to variability. The large bandwidths increase variability which is why the Type 1 error rate begins to inflate using the adaptive-smoothing critical values. The fixed-smoothing critical values are bandwidth specific and account for the distributional effects of $\hat{\Omega}_T$. Type 1 error does not suffer from increasing the bandwidth while using fixed-smoothing critical values, but the Type 2 error is worse for fixed-smoothing critical values then for adaptive-smoothing. Accounting for the increased variability inflates the critical values and the distributions become hard to distinguish. The power is essentially the same for all scenarios, as each procedure yielded a nearly equivalent estimate, and the fixed-smoothing and adaptive-smoothing critical values converge.

For $\rho = .75$, the simulation results are presented in the bottom rows of Figures~\ref{fig:performance1} and~\ref{fig:performance4} for $d=1$ and $d=4$, respectively.  Here we observe an arm like shape for Type 1 error as expected.  Zero lugsail kernel with fixed-smoothing critical value essentially attains the prescribed error rate. We further observe using fixed-smoothing critical values the zero lugsail kernel is closer to the prescribed error rate for nearly all values of $b$. The power curve also changes in the presence of correlation. Difference between the dashed lines and solid lines between the mother and zero lugsail settings are due to variability, and the difference between the colors are due to bias. The larger the bandwidth, the more the curves diverge from each other as expected. The difference-based estimator has the benefit of not requiring a particular bandwidth and also exhibits the best power curve. However, the Type 1 error rate suffers as the number of dimensions or strength of correlation increase. The difference between the observed versus objective Type 1 error rate is much wider than the optimal versions of the other procedures for the correlated settings. 

\subsection{Comparing Procedures Mechanics}

In this section we compare the Type 1 error rates for a variety of models. Recall $\{w_t\}$ be an unknown stationary process with mean zero and finite variance that is possibly correlated or heteroskedastic. We further have
\begin{align*}
        x_t' = & (1 , x_{t, 1} , \dots, x_{t, p}) & \theta' = & (\gamma, \theta_1, \dots, \theta_p)\\
       X' = & [x_1', x_2', \dots, x_T']  &   Y' =&[ y_1, \dots, y_T] 
\end{align*}
where $\theta \in \mathbb{R}^{p+1}$ is a vector of parameters to be estimated, $\{y_t\}$ is a sequence of observed random variables, and $X$ is an observed $(T \times (p+1))$ matrix. For our simulation study we use the common linear model structure: $ y_t = x_t' \theta + w_t$, where  $t = 1, \dots, T$. To estimate $\theta$ we use the OLS estimators at \eqref{eq:ols_theta}.

To conduct inference on the parameters $\theta$, we require the covariance matrix for $\theta$. We first construct $\hat{\Omega}_T$ from \eqref{eq:LRV_centered_error} using $\hat{u}_t= x_t y_t$ and 
 $\kappa \in \mathscr{K}_2$. Let $M = (\frac{1}{T}X'X)^{-1}$, then the covariance matrix of $\theta$ is $M \hat{\Omega}_T M$ where the $i^{th}$ diagonal element is the variance $i^{th}$ element of $\theta$. Let  $e_t \overset{iid}{\sim} N_p(0, \mathbf{1}_p)$. We generate $X$ using an auto-regressive processes of the following form, 
\begin{equation*}
    x_{t,i} =  \rho x_{t-1,i}+ e_t \; . \label{eq:x_AR1}
\end{equation*}
That is, each $x_{t,i}$ is an independent AR(1) process. We consider the following different structures for the error term $w_t$. 

\begin{enumerate}
    \item \textbf{AR1-HOMO}: An auto-regressive model of order 1 with homoskedastic disturbances. Let $\epsilon_t \overset{iid}{\sim} N(0,1)$ and $w_t = \rho_w w_{t-1} + \epsilon_t$.
    \item \textbf{AR1-HET}: An auto-regressive process of order 1 with heteroskedastic disturbances. Let $a_0 = 5$, $a_1 = .25$, $\epsilon_t \overset{iid}{\sim} N(0,1)$, and $w_t = \sqrt{a_0 + a_1 w_{t-1}}v_t$ where $v_t = \rho_v v_{t-1}+ \epsilon_t$.
    \item \textbf{ARMA-G}: A simple auto regressive moving average model, ARMA(1,1), with standard Gaussian disturbances.  Let $Z_t \overset{iid}{\sim} N(0, 1)$, $\epsilon_t \overset{iid}{\sim} N(0,1)$, $\psi_z = 0.5$, and 
    \begin{align}
        w_t =&  Z_t + \psi_z Z_{t-1} + \rho_w w_{t-1} + \epsilon_t. 
        \label{eq:error_ARMA}
    \end{align}
    \item \textbf{ARMA-L}: A simple auto regressive moving average model, ARMA(1,1), with standard Laplace disturbances. Let $Z_t$ and $\psi_z = 0.5$ be as defined in \eqref{eq:error_ARMA}, $\epsilon_t \overset{iid}{\sim} Laplace(0,1)$, and 
    \begin{align*}
        w_t =&  Z_t + \psi_z Z_{t-1} + \rho_w w_{t-1} + \epsilon_t. 
    \end{align*}
\end{enumerate}

In general we consider a range of values and let $0<\rho = \rho_w = \rho_v <1$. The first model AR1-HOMO is self evident, it is used as the backbone of the bandwidth rule for our procedure, and of most procedures in this space \citep[see e.g.][]{andrews1991heteroskedasticity, vats2022lugsail}.  We next consider the model AR1-HET, a commonly used model for errors with serial correlation and heteroskedasticy \citep{andrews1991heteroskedasticity, hamilton1994time, wool2015}.  The last two models are ARMA-L and ARMA-G are used to explore different error processes with Gaussian and non-Gaussian disturbances \citep{mcelroy2022estimating, lazarus2016har}. 

Our test is invariant to our choice of $\theta$ so we generate the data with $\gamma= 0 $ and $\theta_i = 0$ for $i = 1, \dots, p$. We test the hypotheses that the first coefficient after the intercept is zero, i.e. $H_0: \theta_1 = 0$, using $\alpha = 0.05$ and the Bartlett kernel. The test statistic can be expressed as
\begin{align*}
   F_{OLS} = &  \left [ \sqrt{T} \left(\hat{\theta}_{1,T} - \theta_{1,0} \right) \right ]'\left[M \hat{\Omega}_T^{-1} M\right ]_{22}\left [ \sqrt{T} \left(\hat{\theta}_{1,T} - \theta_{1,0} \right) \right ]/d
   = \frac{T \hat{\theta}_{1,T}^2}{Var(\hat{\theta}_1)} \; .
\end{align*} 
Different sample sizes, correlation coefficients, and procedures are considered. We use the zero lugsail kernel and the absolute value of $\hat{\rho}$ to obtain $b_{opt}$. We also consider the bandwidth rule $b_{ft}$ proposed by \cite{politis2003adaptive}, which works for the specific zero lugsail settings here. We use the $b_{ft}$ rule with a zero lugsail kernel using both fixed-smoothing and adjusted-smoothing critical values as this procedure depends on the behavior of the autocovariance function alone and not testing mechanics. Lastly, we consider the classical procedure with the mother kernel and adjusted-smoothing critical values. Results are presented in Tables~\ref{AR_HOMO}, \ref{AR_HET}, \ref{ARMA_G}, and~\ref{ARMA_L}.

\begin{table}[hbt]
\centering
\begin{tabular}{lc||cc|cc}
 \hline 
 & & \multicolumn{2}{c|}{Fixed-Smoothing} & \multicolumn{2}{c}{Adjusted-Smoothing}\\
 & $\rho$ & $b_{opt}$ & $b_{ft}$  & $b_{ft}$  & $b_{mse}$  \\ 
 \hline 
 & 0.00 & 0.044* & 0.045 & 0.049+ & 0.045 \\ 
 & 0.25 & 0.079 & 0.068*+ & 0.072 & 0.069 \\ 
  {\emph{ T=200}} & 0.50& 0.089 & 0.076*+ & 0.079 & 0.086 \\ 
 & 0.75 & 0.104*+ & 0.148 & 0.159 & 0.117 \\ 
 & 0.90& 0.151*+ & 0.189 & 0.212 & 0.193 \\ 
   \hline
 &0.00& 0.062*+ & 0.066 & 0.067 & 0.063 \\ 
 & 0.25 & 0.067 & 0.061*+ & 0.063 & 0.063 \\ 
  {\emph{ T=500}} & 0.50& 0.061 & 0.063 & 0.065 & 0.057*+ \\ 
 & 0.75 & 0.087*+ & 0.104 & 0.106 & 0.091 \\ 
& 0.90& 0.111*+ & 0.119 & 0.123 & 0.124 \\ 
   \hline
 &0.00& 0.035* & 0.037 & 0.038+ & 0.035 \\ 
   & 0.25 & 0.069 & 0.058*+ &0.060& 0.058 \\ 
  {\emph{ T=1000}} & 0.50&0.060& 0.063 & 0.064 & 0.058*+ \\ 
   & 0.75 & 0.061 & 0.072 & 0.077 & 0.057*+ \\ 
   & 0.90& 0.089*+ & 0.096 & 0.104 & 0.098 \\ 
\hline
  \end{tabular}
  \caption[Type 1 Error Rate for the AR-HOMO Model.]{Results with AR(1)-homoskadiscity error terms. The $*$ denotes the smallest Type 1 error rate, and the $+$ denotes the setting closes to the prescribed Type 1 error rate $\alpha = 0.05$.}
  \label{AR_HOMO}
\end{table}

\begin{table}[hbt]
\centering
\begin{tabular}{lc||cc|cc}
\hline
 & & \multicolumn{2}{c|}{Fixed-Smoothing} & \multicolumn{2}{c}{Adjusted-Smoothing}\\

 & $\rho$ & $b_{opt}$ & $b_{ft}$  & $b_{ft}$  & $b_{mse}$  \\ 
\hline
 {\emph{          }} &0.00& 0.056 & 0.054 & 0.058 & 0.053*+ \\ 
  {\emph{           }} & 0.25 & 0.086 & 0.074*+ &0.080& 0.075 \\ 
  {\emph{ T=200}} & 0.50& 0.087 & 0.073*+ & 0.077 & 0.078 \\ 
  {\emph{  }} & 0.75 & 0.087*+ & 0.119 & 0.123 & 0.101 \\ 
  {\emph{   }} & 0.90& 0.157*+ & 0.198 & 0.218 & 0.214 \\ 
   \hline
{\emph{    }} & 0.00& 0.067 &0.070& 0.072 & 0.066*+ \\ 
  {\emph{            }} & 0.25 & 0.054 & 0.050+ & 0.051 & 0.049* \\ 
  {\emph{ T=500}} & 0.50& 0.058 &0.060& 0.061 & 0.055*+ \\ 
  {\emph{     }} & 0.75 & 0.070*+ & 0.099 & 0.103 &0.080\\ 
  {\emph{      }} & 0.90& 0.097*+ & 0.108 & 0.116 & 0.118 \\ 
   \hline
{\emph{       }} & 0.00& 0.047 & 0.046* & 0.050+ & 0.049 \\ 
  {\emph{                 }} & 0.25 & 0.072 & 0.061*+ & 0.062 & 0.063 \\ 
  {\emph{ T=1000}} & 0.50& 0.068*+ & 0.071 & 0.072 & 0.071 \\ 
  {\emph{        }} & 0.75 & 0.071*+ & 0.082 & 0.083 & 0.076 \\ 
  {\emph{             }} & 0.90& 0.072*+ & 0.082 &0.090& 0.083 \\ 
\hline
  \end{tabular}
  \caption[Type 1 Error Rate for the AR-HET Model.]{Results with AR(1)-heteroskedasticy error terms. The $*$ denotes the smallest Type 1 error rate, and the $+$ denotes the setting closes to the prescribed Type 1 error rate $\alpha = 0.05$.}
  \label{AR_HET}
\end{table}

\begin{table}[hbt]
\centering
\begin{tabular}{lc||cc|cc}
 \hline
 & & \multicolumn{2}{c|}{Fixed-Smoothing} & \multicolumn{2}{c}{Adjusted-Smoothing}\\

 & $\rho$ & $b_{opt}$ & $b_{ft}$ & $b_{ft}$  & $b_{mse}$  \\ 
\hline
 & 0.00& 0.052*+ & 0.057 & 0.063 & 0.054 \\ 
 & 0.25 & 0.084 & 0.067*+ & 0.068 & 0.069 \\ 
  {\emph{ T=200}} & 0.50& 0.062 & 0.057*+ & 0.066 & 0.062 \\ 
 & 0.75 & 0.092*+ & 0.132 &0.140& 0.113 \\ 
 & 0.90& 0.142*+ & 0.175 & 0.203 & 0.197 \\ 
   \hline
 & 0.00& 0.056 & 0.051 & 0.056 & 0.05*+ \\ 
 & 0.25 & 0.075 & 0.068 & 0.069 & 0.067*+ \\ 
  {\emph{ T=500}} & 0.50& 0.069*+ & 0.075 & 0.076 & 0.071 \\ 
 & 0.75 & 0.062*+ &0.070& 0.075 & 0.071 \\ 
& 0.90& 0.093*+ & 0.105 & 0.126 & 0.128 \\ 
   \hline
 & 0.00& 0.051+ & 0.057 & 0.057 & 0.049* \\ 
& 0.25 & 0.048+ & 0.037* & 0.039 & 0.039 \\ 
  {\emph{ T=1000}} & 0.50& 0.049*+ & 0.061 & 0.062 & 0.054 \\ 
 & 0.75 & 0.061 & 0.071 & 0.072 & 0.06*+ \\ 
 & 0.90& 0.082*+ & 0.091 & 0.097 & 0.109 \\ 
\hline
  \end{tabular}
  \caption[Type 1 Error Rate for the ARMA-G Model.]{Results with ARMA(1,1) error terms with Gaussian disturbances. The $*$ denotes the smallest Type 1 error rate, and the $+$ denotes the setting closes to the prescribed Type 1 error rate $\alpha = 0.05$.}
  \label{ARMA_G}
\end{table}

\begin{table}[hbt]
\centering
\begin{tabular}{lc||cc|cc}
 \hline 
 & & \multicolumn{2}{c|}{Fixed-Smoothing} & \multicolumn{2}{c}{Adjusted-Smoothing}\\
 & $\rho$ & $b_{opt}$ & $b_{ft}$   & $b_{ft}$  & $b_{mse}$  \\ 
\hline
 & 0.00 & 0.056 & 0.060 & 0.066 & 0.053*+ \\ 
 & 0.25 & 0.067 & 0.061 & 0.065 &0.060*+ \\ 
  {\emph{ T=200}} & 0.50 & 0.076 & 0.069*+ & 0.075 & 0.072 \\ 
& 0.75 & 0.084*+ & 0.132 & 0.136 & 0.106 \\ 
& 0.90 & 0.148*+ & 0.181 & 0.209 & 0.199 \\ 
   \hline
& 0.00 & 0.081 & 0.078 & 0.079 & 0.077*+ \\ 
 & 0.25 & 0.064 & 0.053*+ & 0.053 & 0.055 \\ 
  {\emph{ T=500}} & 0.50 & 0.063*+ & 0.068 & 0.069 & 0.067 \\ 
 & 0.75 & 0.084*+ & 0.102 & 0.108 & 0.093 \\ 
& 0.90 & 0.116*+ & 0.129 & 0.143 & 0.146 \\ 
   \hline
 & 0.00 & 0.054 & 0.057 & 0.058 & 0.052*+ \\ 
 & 0.25 & 0.066 & 0.053*+ & 0.054 & 0.054 \\ 
  {\emph{ T=1000}} & 0.50 & 0.059*+ & 0.069 & 0.072 & 0.060 \\ 
 & 0.75 & 0.066*+ & 0.073 & 0.077 & 0.069 \\ 
 & 0.90 & 0.065*+ & 0.072 & 0.079 & 0.079 \\ 
 \hline
  \end{tabular}
  \caption[Type 1 Error Rate for the ARMA-L Model.]{Results with ARMA(1,1) error terms with Laplace disturbances. The $*$ denotes the smallest Type 1 error rate, and the $+$ denotes the setting closes to the prescribed Type 1 error rate $\alpha = 0.05$.}
  \label{ARMA_L}
\end{table}

Across the different settings, the observed Type 1 error rate closest to the prescribed rate was typically the procedure that relied on $b_{opt}$. In particular, for $\rho = 0.9$ the bandwidth rule $b_{opt}$ prevailed in all instances. When correlation was low the rates where more comparable as expected. Interestingly, we observe a larger difference between the procedures for the models AR1-HET and ARMA-L, with the $b_{opt}$ prevailing. This suggests the $b_{opt}$ procedure is more robust to complicated error structures. In instances when $b_{opt}$ is not the optimal choice it is typically near it. 

\subsection{Discussion}
\label{chap_sec:discussion}

Through our simulation study we observe the behavior of Type 1 error, Type 2 error, and power. The fixed-smoothing asymptotic framework accounts for the variability that causes the Type 1 error rate to inflate as $b$ increases despite $\hat{\Omega}_T$ being less biased. This is evident by the dashed lines in Figures~\ref{fig:performance1} and~\ref{fig:performance4}. The zero lugsail estimator in contrasts addresses the asymptotic bias and has smaller size distortion in comparison to the mother lugsail settings at the same bandwidth, as expected. The testing mechanics work well as they address different issues.

Our recommended procedure was typically optimal or near optimal, especially in high correlation settings. Additional simulations with various settings showed even larger performance differences, indicating that our recommended procedure is comparatively more robust to complex error structures. We observed that, for nearly any set bandwidth, the recommended procedure results in smaller size distortion. However, since we select $b$ specific to the estimator and setting, such a comparison is not as informative.

There is a distinct and noteworthy advantage in regards to Type 1 error when using a zero lugsail estimator. The cost of this improvement is felt in power loss, which makes selecting an appropriate bandwidth imperative. The bandwidth rule we propose lands at a conservative point, and incorporates multiple facets known to affect testing performance. The procedure proposed here uses minor modifications to current practices: the zero lugsail kernel function, fixed-smoothing critical values, and the zero lugsail bandwidth rule. When used together these small adjustments can have a major impact in improving well known testing performance issues for robust inference procedures.


\clearpage

\begin{appendix}

\section{Proof of Lemma~\ref{lemma:wald_limiting}}

The prove of Lemma~\ref{lemma:wald_limiting} builds upon results in \cite{kiefer2005new, sun2014let}. 

\begin{proof}[Proof of Lemma \ref{lemma:wald_limiting}]

\noindent Under Assumption 2 we have the following Taylor Series approximation
\begin{align*}
    g_T(\hat{\theta}_T) = &  g_T(\theta_0) + G_T (\theta_0) (\hat{\theta}_T - \theta_0) +o_p(1).\\
    \intertext{Multiply both sides of the equation by $G_T(\hat{\theta}_T)' W_T$, }
     G_T(\hat{\theta}_T)' W_Tg_T(\hat{\theta}_T) =& G_T(\hat{\theta}_T)' W_T g_T(\theta) + G_T(\hat{\theta}_T)' W_TG_T ( \theta_0) (\hat{\theta}_T - \theta_0)+o_p(1).\\
     \intertext{Observe that we pick $\hat{\theta}_T$ such that it minimizes $g_T(\hat{\theta}_T)W_T g_T(\hat{\theta}_T)$.  Alternatively, we can think of $\hat{\theta}_T$ as the estimator such that $G_T'(\hat{\theta}_T)W_T g_T(\hat{\theta}_T) = 0$ by using the chain rule, }
     0 = & G_T(\hat{\theta}_T)' W_T g_T(\theta_0) + G_T(\hat{\theta}_T)' W_TG_T( \theta_0) (\hat{\theta}_T - \theta_0)+o_p(1).\\
     \intertext{Solve for $(\hat{\theta}_T- \theta_0)$ and multiply both sides by $T^{1/2}R$,}
     T^{1/2}R(\hat{\theta}_T- \theta_0) = &   -R\left [ G_T(\hat{\theta}_T)' W_TG_T (\theta_0) \right ]^{-1}G_T(\hat{\theta}_T)' W_T T^{1/2}g_T(\theta_0)+o_p(1).
     \intertext{Under Assumption 2 and Assumption 3 we have the following result,}
    T^{1/2} R(\hat{\theta}_T - \theta_0) = & -R(G_0' W_\infty G_0)^{-1} G_0 W_\infty f(v_t; \theta_0) + o_p(1)\\
    = & \frac{1}{T^{1/2}}\sum_{t=1}^T u_t+ o_p(1).
\end{align*}

\noindent Under assumption 4, this implies $T^{1/2} R(\hat{\theta}_T - \theta_0) \overset{d}{\rightarrow}  \Lambda B_p(1).$

\end{proof}




\section{Proof of Proposition~\ref{prop:fixed_b_forms}}

To prove Proposition~\ref{prop:fixed_b_forms} we first introduce Lemma~\ref{lemma:mu_components} to get an expression for $E[v_{11\cdot 2}]$ and $E[v_{11\cdot 2}^2]$. Results in this section build upon those obtained by \cite{sun2014let}. 

\begin{lemma}
    Let $\kappa \in \mathscr{K}_2$ and it's corresponding center kernel be $\tilde{\kappa}$. As $b \rightarrow 0$ and $T \rightarrow \infty$, we have 
    \begin{enumerate}
        \item[(a)] $\mu_1 :=  \int^1_0 \tilde{\kappa}_b(t, t) dt =  1 - bc_1 + O(b^2)$
         \item[(b)] $\mu_2 := \int^1_0 \int^1_0 \left [\tilde{\kappa}_b(t, s)\right ]^2 dt ds  =  \sum_{n=1}^\infty (\lambda_m)^2 =  bc_2 + O(b^2). $
    \end{enumerate}
    \label{lemma:mu_components}
\end{lemma}

\begin{proof}[Proof of Lemma \ref{lemma:mu_components}a]

\noindent Define
\begin{align*}
    \mu_1 := & \int_0^1 \tilde{\kappa}_b(t, t) dt\\
    = & 1 - \int_0^1 \int_0^1 \kappa_b(t, s)dt ds. \\
   \intertext{Note that $\int^1_0  \left( \int^1_0 \kappa_b(t-r) dr \right)^2 dt = \int^1_0  \int^1_0 \int^1_0 \kappa_b(t-r) \kappa_b(t-s) dr ds dt$. We use this property to simplify the expression for $\mu_2$, defined as}
    \mu_2 := & \int_0^1 \int_0^1 [\tilde{\kappa}_b(t, s)]^2 dtds\\
    = & \left (\int_0^1 \int_0^1 \kappa_b(t-s) dt ds \right)^2 + \int_0^1 \int_0^1 \kappa_b^2(t-s) dt ds \\
    &  -2 \int_0^1 \int_0^1 \int_0^1 \kappa_b(t-p) \kappa_b(t-q) dt dq dp.
\end{align*}
\noindent  Define the following Fourier transformations and the corresponding inverse transformations, 
\begin{align*}
    K_1(\lambda) = & \frac{1}{2\pi} \int^\infty_{-\infty} \kappa\left ( \frac{x}{b} \right ) \exp \left( \frac{-i \lambda x}{b} \right ) dx &  \kappa\left ( \frac{x}{b} \right )  = & \int^\infty_{-\infty} K_1(\lambda) \exp \left( \frac{i \lambda x}{b} \right ) d\lambda \\
    K_2(\lambda) = & \frac{1}{2\pi} \int^\infty_{-\infty} \kappa^2\left ( \frac{x}{b} \right ) \exp \left( \frac{-i \lambda x}{b} \right ) dx & \kappa^2\left ( \frac{x}{b} \right ) = &  \int^\infty_{-\infty} K_2(\lambda) \left( \frac{i \lambda x}{b} \right ) d\lambda. 
\end{align*}
\noindent Using the inverse Fourier transformations we derive an expression for $\int^1_0 \int^1_0 \kappa_b(t-s) dt ds$, 
\begin{align*}
    \int^1_0 \int^1_0 \kappa_b(t-s) dt ds = & \int^1_0 \int^1_0 \int_{-\infty}^\infty K_1(\lambda) \exp \left (-i\lambda \left (\frac{t}{b} - \frac{s}{b} \right ) \right ) dt ds\\
     = & \int^1_0 \int^1_0 \int_{-\infty}^\infty K_1(\lambda) \exp \left (-i\lambda \frac{t}{b} \right )\exp \left (i\lambda \frac{s}{b} \right ) dt ds d\lambda   \\
     = &  \int_{-\infty}^\infty K_1(\lambda)\left [\int_0^1\exp \left (-i\lambda \frac{t}{b} \right )dt\right]\left [\int_0^1\exp \left (i\lambda \frac{s}{b} \right ) ds \right] d\lambda.
\end{align*}
\noindent We integrate the inside terms, use Euler's formula, and properties of imaginary numbers, 
\begin{align*}
     \int^1_0 \int^1_0 \kappa_b(t-s) dt ds = &  \int_{-\infty}^\infty K_1(\lambda)\left [\frac{-i \left ( \exp(i\lambda/b) -1\right )}{\left ( \lambda/b \right) }\right]\left [\frac{\sin(\lambda/b) + i\left( \cos(\lambda/b) -1 \right) }{(\lambda/b)}\right] d\lambda\\
     = &  \int_{-\infty}^\infty K_1(\lambda)\left(\frac{b}{\lambda} \right)^2\left [-i\left ( \cos(\lambda/b) + i \sin(\lambda/b) -1\right) \right]\\
     &  \times \left [\sin(\lambda/b) + i\left( \cos(\lambda/b) -1 \right) \right] d\lambda\\
     = &  \int_{-\infty}^\infty K_1(\lambda)\left(\frac{b}{\lambda} \right)^2\left [-i\left ( \cos(\lambda/b)  -1\right)+ \sin(\lambda/b) \right]\\
     &  \times \left [\sin(\lambda/b) + i\left( \cos(\lambda/b) -1 \right) \right] d\lambda\\
      = &  \int_{-\infty}^\infty K_1(\lambda)\left(\frac{b}{\lambda} \right)^2\left [\left \{i\left ( \cos(\lambda/b)  -1\right)\right \}^2+ \sin^2(\lambda/b)  \right] d\lambda\\ 
      = &  \int_{-\infty}^\infty K_1(\lambda)\left(\frac{b}{\lambda} \right)^2\left [\left ( 1-\cos(\lambda/b)\right)^2+ \sin^2(\lambda/b)  \right] d\lambda\\ 
      = &  \int_{-\infty}^\infty K_1(\lambda)\left(\frac{b}{\lambda} \right)^2\left [\left ( \cos^2(\lambda/b) - 2\cos(\lambda/b)+1\right) + \sin^2(\lambda/b)  \right] d\lambda.
\end{align*}
\noindent We further use the identities to continue to simplify,
\begin{align*}
      \int^1_0 \int^1_0 \kappa_b(t-s) dt ds = &  \int_{-\infty}^\infty K_1(\lambda)\left(\frac{b}{\lambda} \right)^2\left [2 - 2\cos(\lambda/b)   \right] d\lambda\\
      = &  \int_{-\infty}^\infty 2 K_1(\lambda)\left(\frac{b}{\lambda} \right)^2\left [1 - 1\cos(\lambda/b)   \right] d\lambda\\ 
      = &  \int_{-\infty}^\infty 4 K_1(\lambda)\left(\frac{b}{\lambda} \right)^2 \sin^2\left(\frac{\lambda}{2b}\right)  d\lambda\\ 
     = &  \int_{-\infty}^\infty  4 b^2 K_1(\lambda)\left ( \frac{\sin\left(\frac{\lambda}{2b}\right)}{\lambda} \right )^2  d\lambda\\ 
     = &  \int_{-\infty}^\infty  \left [K_1(\lambda) - K_1(0) + K_1(0)\right ] 4 b^2 \left ( \frac{\sin\left(\frac{\lambda}{2b}\right)}{\lambda} \right )^2  d\lambda\\ 
     = &  \int_{-\infty}^\infty  K_1(0) 4 b^2 \left ( \frac{\sin\left(\frac{\lambda}{2b}\right)}{\lambda} \right )^2  d\lambda \\
     & +  \int_{-\infty}^\infty  \left [K_1(\lambda) - K_1(0)\right ] 4 b^2 \left ( \frac{\sin\left(\frac{\lambda}{2b}\right)}{\lambda} \right )^2  d\lambda\\ 
     = &  K_1(0) 4 b^2  \frac{\pi}{2b}+  \int_{-\infty}^\infty  \left [K_1(\lambda) - K_1(0)\right ] 4 b^2 \left ( \frac{\sin\left(\frac{\lambda}{2b}\right)}{\lambda} \right )^2  d\lambda \\   
    = &  K_1(0) 2 b \pi  +  \int_{-\infty}^\infty  \left [K_1(\lambda) - K_1(0)\right ] 4 b^2 \left ( \frac{\sin\left(\frac{\lambda}{2b}\right)}{\lambda} \right )^2  d\lambda. \stepcounter{equation}\tag{\theequation}\label{Lemma2:simplified}
\end{align*}
\noindent Now we concentrate on the last component of (\ref{Lemma2:simplified}) and drop the constants. Using identities and rearranging terms, 
\begin{align*}
    \intertext{$\int_{-\infty}^\infty  \left [\frac{K_1(\lambda) - K_1(0)}{\lambda^2}\right ] \left ( \sin^2\left(\frac{\lambda}{2b}\right) \right ) d\lambda$} 
    = & \int_{-\infty}^\infty  \left [\frac{K_1(\lambda) - K_1(0)}{\lambda^2}\right ] \left ( \sin^2\left(\frac{\lambda}{2b}\right) -\frac{1}{2}+\frac{1}{2}\right ) d\lambda \\
    = & \int_{-\infty}^\infty  \left [\frac{K_1(\lambda) - K_1(0)}{\lambda^2}\right ] \left ( \sin^2\left(\frac{\lambda}{2b}\right) -\frac{1}{2}\right ) d\lambda +  \int_{-\infty}^\infty  \left [\frac{K_1(\lambda) - K_1(0)}{\lambda^2}\right ] \left ( \frac{1}{2}\right ) d\lambda \\
    = & \int_{-\infty}^\infty  \left [\frac{K_1(\lambda) - K_1(0)}{\lambda^2}\right ] \left ( \cos\left(\frac{\lambda}{b}\right)\right ) d\lambda +   \left ( \frac{1}{2}\right )\int_{-\infty}^\infty  \left [\frac{K_1(\lambda) - K_1(0)}{\lambda^2}\right ] d\lambda. \stepcounter{equation}\tag{\theequation}\label{Lemma2:ReimannLegesueNeeded}
\end{align*}
\noindent Recall the Riemann-Lebesgue Lemma \citep{folland1999real}. Observe $ K_1(0) = \frac{c_1}{2\pi}$ is a constant, $K_1(\lambda)$ is finite and even, and thus consequentially $\frac{K_1(\lambda) - K_1(0)}{\lambda^2}$ is a finite even function.  We further observe that 
\begin{equation*}
        \int_{-\infty}^\infty  \left [\frac{K_1(\lambda) - K_1(0)}{\lambda^2}\right ] \left ( \cos\left(\frac{\lambda}{b}\right)\right ) d\lambda
\end{equation*}
\noindent is the cosine representation of Fourier transformations for even functions.  Thus, we expect it to tend towards 0 by Riemann-Lebesgue Lemma as $\frac{1}{b} \rightarrow \infty$ (or $b \rightarrow 0$). Therefore we have the following, 
\begin{align}
    \int^1_0 \int^1_0 \kappa_b(t-s) dt ds =& K_1(0) 2 b \pi  +   \left ( \frac{4b^2}{2}\right )\int_{-\infty}^\infty  \left [\frac{K_1(\lambda) - K_1(0)}{\lambda^2}\right ] d\lambda + o(b^2). \nonumber\\
\end{align}
\noindent Again by recognizing properties of Fourier transformations of even functions we can substitute:  $K_1(\lambda) = (2\pi)^{-1} \int^\infty_{-\infty} \kappa \left (\frac{x}{b}\right )  \cos \left (\lambda \frac{x}{b} \right ) dx$,
\begin{align}
   \int^1_0 \int^1_0 \kappa_b(t-s) dt ds = & K_1(0) 2 b \pi  +   \left ( \frac{2b^2}{\pi}\right )\int_{-\infty}^\infty \int_{-\infty}^\infty \kappa \left (\frac{x}{b}\right ) \left (\frac{cos\left (\frac{\lambda x}{b}\right ) - c_1}{\lambda^2}\right ) dx d\lambda + o(b^2). \label{lemma2:cases_c1}
\end{align}
\noindent We now consider three cases for $c_1$: $c_1 = 1$, $c_1  < 1$, and $c_1 >1$. We first consider the case when $c_1 = 1$. Using identities and properties of mother kernels we obtain
\begin{align*}
     \int^1_0 \int^1_0 \kappa_b(t-s) dt ds  =  & K_1(0) 2 b \pi  +   \left ( \frac{2b^2}{\pi}\right )\int_{-\infty}^\infty \int_{-\infty}^\infty \kappa\left (\frac{x}{b}\right )\frac{sin^2\left (\frac{\lambda}{2}\frac{x}{b}\right )}{\lambda^2} dx d\lambda + o(b^2)\\
       =  & K_1(0) 2 b \pi  +   \left ( \frac{2b^2}{\pi}\right )\int_{-\infty}^\infty \int_{-\infty}^\infty \kappa(y)\frac{sin^2\left (\frac{\lambda}{2}y\right )}{\lambda^2} dy d\lambda + o(b^2)\\
       =  & K_1(0) 2 b \pi  +   \left ( \frac{2b^2}{\pi}\right )\int_{-\infty}^\infty \int_{-\infty}^\infty \kappa(y) \frac{sin^2(\lambda y/2)}{\lambda^2} dy d\lambda + o(b^2)\\
       =  & K_1(0) 2 b \pi  +   b^2\int_{-\infty}^\infty \int_{-\infty}^\infty \kappa(y) |y| dy d\lambda + o(b^2)\\
       \approx  & K_1(0) 2 b \pi  +  O(b^2) + o(b^2)\\
       \approx  & b c_1   +  O(b^2). 
\end{align*}
\noindent Next we consider the case when $c_1 <1$. Observe that $\forall y$ and $\forall \lambda$ 
\begin{align*}
    -\kappa(y) \left(\frac{cos(\lambda y) -1}{\lambda^2}\right) <\kappa(y) \left(\frac{cos(\lambda y) -c_1}{\lambda^2}\right) < \kappa(y) \left(\frac{cos(\lambda y) -1}{\lambda^2}\right) \\
   -\kappa(y) \left(\frac{sin^2( \lambda y/2) }{\lambda^2}\right)  <\kappa(y) \left(\frac{cos(\lambda y) -c_1}{\lambda^2}\right)<\kappa(y) \left(\frac{sin^2( \lambda y/2) }{\lambda^2}\right)\\
    -\int_{-\infty}^\infty k(y)|y| dy <\int_{-\infty}^\infty \int_{-\infty}^\infty \kappa(y) \left(\frac{cos(\lambda y) -c_1}{\lambda^2}\right)dy d\lambda < \int_{-\infty}^\infty k(y)|y| dy. 
\end{align*}
\noindent Define $s:= \int_{-\infty}^\infty \int_{-\infty}^\infty \kappa(y) \left(\frac{cos(\lambda y) -c_1}{\lambda^2}\right)dy d\lambda < \infty $ and hence 
\begin{align*}
     \int^1_0 \int^1_0 \kappa(t-s) dt ds =& b c_1 - b^2 s + o(b^2) \\
     =& b c_1 + O(b^2). 
\end{align*}
\noindent Lastly for the case when $c_1>1$.  We can rewrite $c_1$ into components $c_1 = c_{1.1} + c_{1.r}$, where $c_{1.1} = 1$ and $c_{1.r} = c_{1} -1$.  Then, 
\begin{align*}
    \intertext{$\left ( \frac{2b^2}{\pi}\right )\int_{-\infty}^\infty \int_{-\infty}^\infty \kappa(y) \left (\frac{cos(\lambda y) - c_1}{\lambda^2}\right ) dy d\lambda + o(b^2) $}
    = & \left ( \frac{2b^2}{\pi}\right )\int_{-\infty}^\infty \int_{-\infty}^\infty \kappa(y) \left (\frac{cos(\lambda y) - c_{1.1} -  c_{1.r}}{\lambda^2}\right ) dy d\lambda + o(b^2)\\
    =&  \left ( \frac{2b^2}{\pi}\right )\int_{-\infty}^\infty \int_{-\infty}^\infty \kappa(y) \left (\frac{cos(\lambda y) - c_{1.1}}{\lambda^2}\right ) dy d\lambda \\
    & - \left ( \frac{2b^2}{\pi}\right )\int_{-\infty}^\infty \int_{-\infty}^\infty \kappa(y) \left (\frac{c_{1.r}}{\lambda^2}\right ) dy d\lambda + o(b^2) \\
    =& O(b^2). 
\end{align*}
\noindent Thus $\mu_1 =  \int^1_0 \tilde{\kappa}_b(t, t) dt = 1 - bc_1 + O(b^2)$ as desired. 

\end{proof}

\begin{proof}[Proof of Lemma \ref{lemma:mu_components}b]

\noindent First we focus on the term $\int_0^1 \int_0^1 \kappa_b^2(t-s) dt ds$.  Using the same arguments as Lemma \ref{lemma:mu_components}a, we can replace all instances of $K_1$ with $K_2$. Recalling the mother kernel property: $\int_{-\infty}^\infty \kappa^2(x) x^2dx < \infty$, we get the following,  
\begin{equation}
    \int_0^1 \int_0^1 \kappa_b(t-s) dt ds = bc_2 + O(b^2). 
    \label{Lemma2b:kappa_sq}
\end{equation}
\noindent  We next consider: $\int_0^1 \int_0^1 \int_0^1 \kappa_b(t-p) \kappa_b(t-q) dt dq dp$. We start with an alternative representation of the familiar integral using the Fourier transforms and inverses from Lemma \ref{lemma:mu_components}a with kernel function $\kappa$, and use properties of Fourier transforms and even functions
\begin{align*}
    \int_0^1 \kappa_b(t-s) ds = & \int_0^1 \left ( \int^\infty_{-\infty} K_1(\lambda) \exp\left(  \frac{-i\lambda(t-s)}{b}\right ) d\lambda \right ) ds\\
    = & \int_0^1 \left ( \int^\infty_{-\infty} K_1(\lambda) \cos\left ( \frac{\lambda(t-s)}{b} \right) d\lambda \right ) ds\\
    = & \int_{-\infty}^\infty \int^1_{0} K_1(\lambda)cos\left ( \left(\frac{\lambda(t-s)}{b} \right ) dx  \right ) d\lambda.\\
    \intertext{We rewrite the integral using a u-substitution with $x = \lambda/b$,}
    \int_0^1 \kappa_b(t-s) ds = & -\int_{-\infty}^\infty K_1(\lambda) \left (\frac{b}{\lambda} \right) \left [ \sin \left (\frac{ \lambda (t-1)}{b} \right ) - \sin \left(\frac{ \lambda t}{b}\right)\right] d\lambda\\ 
    = & -b^2\int_{-\infty}^\infty K_1(xb) \left (\frac{1}{x} \right) \left [ \sin \left (x(t-1) \right )  \sin \left(xt\right)\right] dx. \stepcounter{equation}\tag{\theequation}\label{Lemma2b:sin_rep}
\end{align*}
\noindent We rewrite the triple integral with the expression in equation (\ref{Lemma2b:sin_rep}),
\begin{align*}
    \intertext{$\int_0^1 \int_0^1 \int_0^1 \kappa_b(t-p) \kappa_b(t-q) dt dq dp$}
     = & \int_0^1\left [ \int_0^1 \ \kappa_b(t-p)  dp \right ] \left[ \int_0^1\kappa_b(t-q)dq \right] dt\\
    = & \int_0^1\left [ \int_0^1 \ \kappa_b(t-p)  dp \right ]^2 dt\\
    = & b^2\int_0^1\left [ \int_{-\infty}^\infty K_1(xb) \left (\frac{1}{x} \right) \left [ \sin \left (x (t-1) \right ) - \sin \left( xt\right)\right] dx\right ]^2 dt\\
    = &b^2 \int_0^1\left [ \int_{-\infty}^\infty K_1(xb) \left (\frac{1}{x} \right) \sin \left (x (t-1) \right ) dx-  \int_{-\infty}^\infty K_1(xb)\left (\frac{1}{x} \right)\sin \left(x t\right) dx\right ]^2 dt.\\
    \intertext{For each $x$, $K_1(x b) \approx K_1(0) + o(1) = K_1(0)(1+ o(1))$ as $b \rightarrow 0$,}
    \intertext{$\int_0^1 \int_0^1 \int_0^1 \kappa_b(t-p) \kappa_b(t-q) dt dq dp$}
    = & b^2 K_1^2(0)\int_0^1\left [ \int_{-\infty}^\infty  \frac{1}{x} \sin \left (x (t-1) \right ) dx-  \int_{-\infty}^\infty \frac{1}{x}\sin \left( x t\right) dx\right ]^2 dt(1+ o(1)).\\
    \intertext{We further use a u-substitution, the property that $K_1(0) = \frac{c_1}{2\pi}$ to simplify terms, }
    \intertext{$\int_0^1 \int_0^1 \int_0^1 \kappa_b(t-p) \kappa_b(t-q) dt dq dp$}
    = & b^2 K_1^2(0) \int_0^1 \left [2\int_{-\infty}^\infty \frac{1}{y}sin \left(y\right) dy \right ]^2 dt(1+ o(1)) \\ 
    = & b^2 c_1^2 (2\pi)^{-2} \int_0^1 \left [2\int_{-\infty}^\infty \frac{1}{y}sin \left(y\right) dy \right ]^2 dt(1+ o(1)) \\  
    = & b^2 c_1^2 (2\pi)^{-2} \int_0^1 2^2 \pi^2 dt(1+ o(1)) \\
    = & b^2c_1^2 + o(b^2). 
\end{align*}

\noindent Therefore, we have the following
\begin{align*}
    \mu_2 = & [bc_1 +O(b^2)]^2 + [bc_2 +O(b^2)]-2[b^2 c_1^2 +O(b^2)]\\
    = & [bc_1]^2 + bc_2 - 2[bc_1]^2 +O(b^2) \\
    = & bc_2 + O(b^2). 
\end{align*}

\end{proof}

With Lemma~\ref{lemma:mu_components} above provided and Lemma 3 in \cite{sun2014let}, we can obtain the expressions $E(v_{11 \cdot2})$ and $E(v_{11 \cdot 2}^2)$.

\section{Proof of Theorem~\ref{thm:null_alt_distr} and Corollary~\ref{corollary:error}}

Proposition~\ref{prop:theta_u_indep}, Lemma~\ref{lemma:chi_sq_xi}, and Lemma~\ref{lemma:taylor_series} are all necessary for the proof of Theorem~\ref{thm:null_alt_distr}. These results build upon and generalize results in \cite{sun2014let, sun2011robust, lazarus2021size}.  

\begin{proposition}
    Under Assumption 5i and 5ii, 
    \begin{equation}
        \hat{\theta}_{GLS} - \theta \perp \hat{u}_t,
    \end{equation}
    for $ t = (1, \dots, T)$.
    \label{prop:theta_u_indep}
\end{proposition}

\begin{proof}[Proof of Proposition \ref{prop:theta_u_indep}]

\noindent  Let $\StackDem{\hat{u}}{dT \times 1} = [\hat{u}_1', \dots, \hat{u}_T]'$. Observe, that
\begin{align*}
    \hat{u}_t & = \StackDem{R}{d \times p}( y_t - \hat{\theta}_{OLS})\\
    & =  R( y_t - \overline{y})\\
    & =  R( v_t - \overline{v})
\end{align*}
\noindent and $(l_T l_T'/T\otimes \mathbb{I}_p)v$ is equal to the $Tp \times 1$ vector $[\overline{v}, \dots \overline{v}]'$. 
 Thus $ \left [ \left (\mathbb{I}_T - l_T l_T'/T\right ) \otimes \mathbb{I}_p \right ]v = [v_1 - \overline{v}, \dots, v_T - \overline{v}]'$ and  $\hat{u} = (R \otimes \mathbb{I}_T)\left [(\mathbb{I}_T - l_T l_T'/T) \otimes \mathbb{I}_p \right ]v$. If $\Cov(\hat{\theta}_{GLS} - \theta, \hat{u}_t) = \StackDem{0}{p \times d}$ for all $t$, then equivalently, $\Cov \left( \left [(l_T \otimes \mathbb{I}_p)' \Omega_v^{-1} (l_T \otimes \mathbb{I}_p) \right ]^{-1} (l_T \otimes \mathbb{I}_p)\Omega_v^{-1}v, \hat{u} \right ) = \StackDem{0}{p \times Td}$. We check this relation using algebraic manipulation, 
\begin{align*}
    \Cov(\hat{\theta}_{GLS} - \theta, \hat{u})= & \Cov \left( \left [(l_T \otimes \mathbb{I}_p)' \Omega_v^{-1} (l_T \otimes \mathbb{I}_p) \right ]^{-1} (l_T \otimes \mathbb{I}_p)\Omega_v^{-1}v,\right .\\
    & \hspace{1cm}\left . (R \otimes \mathbb{I}_T)\left [(\mathbb{I}_T - l_T l_T'/T) \otimes \mathbb{I}_p \right ]v \right ) \\
    = &  E \left\{  \left [(l_T \otimes \mathbb{I}_p)' \Omega_v^{-1} (l_T \otimes \mathbb{I}_p) \right ]^{-1} (l_T \otimes \mathbb{I}_p) \Omega_v^{-1}vv'\left  [(\mathbb{I}_T - l_T l_T'/T) \otimes \mathbb{I}_p \right ]\right\}\\
    & \times (R \otimes \mathbb{I}_T)'+ 0 \\
    = &  E \left\{  \left [(l_T \otimes \mathbb{I}_p)' \Omega_v^{-1} (l_T \otimes \mathbb{I}_p) \right ]^{-1} (l_T \otimes \mathbb{I}_p) \Omega_v^{-1}\Omega_v\left  [(\mathbb{I}_T - l_T l_T'/T) \otimes \mathbb{I}_p \right ]\right\} \\
    & \times (R \otimes \mathbb{I}_T)' \\
    = &  E \left\{\left [(l_T \otimes \mathbb{I}_p)' \Omega_v^{-1} (l_T \otimes \mathbb{I}_p) \right ]^{-1} (l_T \otimes \mathbb{I}_p)' \left  [(\mathbb{I}_T - l_T l_T'/T) \otimes \mathbb{I}_p \right ]\right\}\\
    & \times (R \otimes \mathbb{I}_T)'\\
    = &  E \left\{\left [(l_T \otimes \mathbb{I}_p)' \Omega_v^{-1} (l_T \otimes \mathbb{I}_p) \right ]^{-1}\right\} \\
    & \times (l_T \otimes \mathbb{I}_p)' \left  [(\mathbb{I}_T - l_T l_T'/T) \otimes \mathbb{I}_p \right ](R \otimes \mathbb{I}_T)'\\
    = &  E \left\{\left [(l_T \otimes \mathbb{I}_p)' \Omega_v^{-1} (l_T \otimes \mathbb{I}_p) \right ]^{-1}\right\} \\
    & \times (l_T' \otimes \mathbb{I}_p) \left  [(\mathbb{I}_T - l_T l_T'/T) \otimes \mathbb{I}_p \right ] (R \otimes \mathbb{I}_T)'\\
    = &  E \left\{\left [(l_T \otimes \mathbb{I}_p)' \Omega_v^{-1} (l_T \otimes \mathbb{I}_p) \right ]^{-1}\right\} \\
    & \times \left (l_T' (\mathbb{I}_T - l_T l_T'/T) \right )  \otimes \left  [ \mathbb{I}_p \right ]\ (R \otimes \mathbb{I}_T)'.
\end{align*}
\noindent Observe that $l_T' (\mathbb{I}_T - l_T l_T'/T)$ produces a $1 \times T$ vector with elements $\left [ (1 - \frac{1}{T}) - \sum_{t=1}^{T-1} \frac{1}{T}\right ]= 0$.  Therefore,  $\left (l_T' (\mathbb{I}_T - l_T l_T'/T) \right )  \otimes \left  [ \mathbb{I}_p \right ] =\StackDem{0}{p \times pT}$ and $\Cov(\hat{\theta}_{GLS} - \theta, \hat{u}) = \StackDem{0}{p \times dT}$, as desired.  

\end{proof}

\begin{lemma} 
Under Assumption 4 and 5,
    \begin{enumerate}
        \item[(a)] $P(dF_{GLS} \leq z) = E \left [G_p(z \Xi_T^{-1}) \right ] + O(T^{-1})$
        \item[(b)] $P(dF_{OLS} \leq z) = P(pF_{GLS} \leq z) + O(T^{-1})$
    \end{enumerate}
\noindent where  
    \begin{align*}
        \StackDem{\Xi_T}{1 \times 1} \left ( \hat{\Omega}_T\right ) = & e_T'[\Omega^{1/2} \hat{\Omega}_T^{-1} \Omega^{1/2}]e_T\\
        \StackDem{e_T}{d \times 1} =&  \frac{\Omega^{-1/2}_{GLS} \sqrt{T}R(\hat{\theta}_{GLS} - \theta_0)}{||\Omega^{-1/2}_{GLS}  \sqrt{T}R(\hat{\theta}_{GLS} - \theta_0)||}
    \end{align*}
\noindent and $\Omega_{GLS}$ is the variance of $\sqrt{T}R(\hat{\theta}_{GLS} - \theta_0)$.
    \label{lemma:chi_sq_xi}
\end{lemma}

\begin{proof}[Proof of Lemma \ref{lemma:chi_sq_xi}a]

\noindent Let $ \Xi_T( \hat{\Omega}_T) = \Xi_T$ for ease of notation. Define 	
\begin{align*}
\Upsilon_T :=&  ||\Omega_{GLS}^{-1/2} \sqrt{T}R( \hat{\theta}_{GLS} - \theta_0)||^2, \\
\Xi_{T,GLS} :=&  e_T'\Omega_{GLS}^{1/2}\hat{\Omega}_T^{-1} \Omega_{GLS}^{1/2} e_T.
\end{align*}
\noindent We can rewrite the definition of $F_{GLS}$ using these values, 
\begin{align*}
    dF_{GLS} = & \left [ \sqrt{T} R(\hat{\theta}_{GLS} - \theta_0)\right]' \hat{\Omega}_T^{-1} \left [ \sqrt{T} R(\hat{\theta}_{GLS} - \theta_0)\right]\\
   = & \left \{\left [ \sqrt{T} R(\hat{\theta}_{GLS} - \theta_0)\right]' \Omega_{GLS}^{-1/2}\right \} \Omega_{GLS}^{1/2}\hat{\Omega}_T^{-1} \Omega_{GLS}^{1/2} \left \{\Omega_{GLS}^{-1/2}\left [ \sqrt{T} R(\hat{\theta}_{GLS} - \theta_0)\right]\right \}\\
    = &||\Omega_{GLS}^{-1/2} \sqrt{T}R( \hat{\theta}_{GLS} - \theta_0)|| e_T'\Omega_{GLS}^{1/2}\hat{\Omega}_T^{-1} \Omega_{GLS}^{1/2} e_T  ||\Omega_{GLS}^{-1/2} \sqrt{T}R( \hat{\theta}_{GLS} - \theta_0)||\\
    =  &||\Omega_{GLS}^{-1/2} \sqrt{T}R( \hat{\theta}_{GLS} - \theta_0)||^2 e_T'\Omega_{GLS}^{1/2}\hat{\Omega}_T^{-1} \Omega_{GLS}^{1/2} e_T \\
    = & \Upsilon_T \Xi_{T,GLS}.
\end{align*}
\noindent From Proposition \ref{prop:theta_u_indep} we observe that $\Upsilon_T$ and $\Xi_{T,GLS}$ are independent. Also, 
\begin{equation*}
e_T'\Omega_{GLS}^{1/2}\hat{\Omega}_T^{-1}\Omega_{GLS}^{1/2}e_T = ||e_T'\Omega_{GLS}^{1/2}\hat{\Omega}_T^{-1/2}||^2>0,
\end{equation*}
\noindent thus we can take the inverse of $\Xi_{T,GLS}$. We obtain an expression for $P(dF_{GLS} \leq z)$ using these facts and the tower rule of expectation, 
\begin{align*}
    P(dF_{GLS} \leq z) = & P(\Upsilon_T \Xi_{T,GLS} \leq z) \\
    = &  P(\Upsilon_T \leq z \left ( \Xi_{T,GLS}\right )^{-1}) \\
    = & E \left (G_p( z \left ( \Xi_{T,GLS}\right )^{-1}) \right ) \\
    = & E \left (G_p \left (z \left (e_T'\Omega_{GLS}^{1/2}\hat{\Omega}_T^{-1} \Omega_{GLS}^{1/2} e_T\right )^{-1} \right ) \right ). 
\end{align*}
\noindent  By the mean value theorem, for $\epsilon, M \in \mathbb{R}^+$ and $\epsilon' \in (0, \epsilon)$
\begin{align*}
    \frac{G_d(x + \epsilon T) - G_d(x)}{\epsilon T} = & G'_d(x + \epsilon'T)\\
    G_d(x + \epsilon T) - G_d(x)= & (\epsilon T)^{-1} G'_d(x + \epsilon'T), \\
    \intertext{ and because $|G_d'(x)| < M <\infty$ (bounded function), we further observe, }
    |G_d(x + \epsilon T) - G_d(x)| < & (\epsilon T)^{-1} M =  O(T^{-1}) \\
    \Rightarrow G_d(x + \epsilon T) - G_d(x) = &  O(T^{-1}) \\
    G_d(x + \epsilon T) = & G_d(x) + O(T^{-1}). 
\end{align*}
\noindent Using the mean value theorem result and recognizing $\Omega_{GLS} = \Omega(1+ O(T^{-1}))$, we obtain the final result
\begin{align*}
    P(dF_{GLS} \leq z) =  & E \left (G_p \left (z \left (e_T'\left\{\Omega(1+ O(T^{-1})) \right \}^{1/2}\hat{\Omega}_T^{-1} \left\{\Omega(1+ O(T^{-1})) \right \}^{1/2} e_T\right )^{-1} \right ) \right ) \\
    =&  E \left [G_p \left (z \Xi^{-1}_T\right)  \right ] + O(T^{-1}).
\end{align*}
\end{proof}

\begin{proof}[Proof of Lemma \ref{lemma:chi_sq_xi}b]
\noindent Define the following: 
\begin{itemize}
    \item $\zeta_{1T} = 2(\sqrt{T} R\Delta)' \hat{\Omega}^{-1} \Omega_{T, GLS}^{1/2}e_T$
    \item $\zeta_{2T} = (\sqrt{T} R\Delta)' \hat{\Omega}^{-1}_T (\sqrt{T} R\Delta)$
    \item $\zeta_{T} =  \zeta_{2T} + \zeta_{1T}\sqrt{\Upsilon_T}$
        \item $\Delta = \hat{\theta}_{GLS} - \theta_0 - (\hat{\theta}_{OLS} - \theta_0)  \Rightarrow (\hat{\theta}_{OLS} - \theta_0) = \Delta + ( \hat{\theta}_{GLS} - \theta_0) $.
\end{itemize}
\noindent We begin by deriving an alternative expression for $dF_{OLS}$ in terms of the $dF_{GLS}$, 
\begin{align*}
    dF_{OLS} = & \left [ \sqrt{T} R(\hat{\theta}_{OLS} - \theta_0)\right]' \hat{\Omega}_T^{-1} \left [ \sqrt{T} R (\hat{\theta}_{OLS} - \theta_0)\right] \\
    = &  \left [ \sqrt{T} R(\hat{\theta}_{GLS} - \theta_0) +  \sqrt{T} R\Delta\right]' \hat{\Omega}_T^{-1} \left [ \sqrt{T} R (\hat{\theta}_{GLS} - \theta_0) + \sqrt{T} R\Delta\right] \\
    = &  \left [ \sqrt{T} R(\hat{\theta}_{GLS} - \theta_0)\right]' \hat{\Omega}_T^{-1} \left [ \sqrt{T} R (\hat{\theta}_{GLS} - \theta_0) \right] + \left [ \sqrt{T} R\Delta\right]' \hat{\Omega}_T^{-1} \left [ \sqrt{T} R\Delta\right]\\
    &  + 2\left [ \sqrt{T} R\Delta\right]' \hat{\Omega}_T^{-1} \left [ \sqrt{T} R (\hat{\theta}_{GLS} - \theta_0) \right]\\
    = & dF_{GLS} + \zeta_{2T} + 2(\sqrt{T} R\Delta)' \hat{\Omega}_T^{-1} \Omega_{T, GLS}^{1/2} \frac{\Omega_{T, GLS}^{-1/2} \sqrt{T} R (\hat{\theta}_{GLS} - \theta_0) }{||\Omega_{T, GLS}^{-1/2} \sqrt{T} R(\hat{\theta}_{GLS} - \theta_0)||}\sqrt{\Upsilon_T}\\
    = & dF_{GLS} + \zeta_{2T} + 2(\sqrt{T} R\Delta)' \hat{\Omega}_T^{-1} \Omega_{T, GLS}^{1/2} e_T\sqrt{\Upsilon_T}\\
    = & dF_{GLS} + \zeta_{2T} + \zeta_{1T}\sqrt{\Upsilon_T}.
\end{align*}
\noindent Therefore, 
\begin{align*}
    P \left(dF_{OLS} \leq z \right) = & P(dF_{GLS} + \zeta_{2T} + \zeta_{1T}\sqrt{\Upsilon_T}\leq z) \\
    = & P\left( \Upsilon_T \Xi_{T,GLS} + \zeta_{2T} + \zeta_{1T}\sqrt{\Upsilon_T} \leq z\right).
\end{align*}
\noindent Define $F(a, b, c) = P \left( \Upsilon_T c + b + a\sqrt{\Upsilon_T} \leq z\right)$.  Let
\begin{equation*}
    \begin{matrix}
      F_a = \frac{\partial}{\partial a} F& F_{aa} = \left(\frac{\partial}{\partial a} \right)^2F\\
    F_b = \frac{\partial}{\partial b} F& F_{bb} = \left(\frac{\partial}{\partial b} \right)^2F\\
    F_{ab} = \frac{\partial}{\partial a}\frac{\partial}{\partial b} F.&
    \end{matrix}
\end{equation*}
\noindent Take a multivariate Taylor Series approximation for variables $a$ and $b$, both around 0,
\begin{align}
    F(\zeta_{1T}, \zeta_{2T}, \Xi_{T,GLS}) = &  F(0,0,\Xi_{T,GLS}) \nonumber\\
    & + F_a(0,0, \Xi_{T,GLS}) (\zeta_{1T}) + \frac{1}{2}F_{aa}(0,0, \Xi_{T,GLS})(\zeta_{1T})^2 \nonumber\\
    & + F_b(0,0,\Xi_{T,GLS})\zeta_{2T} + \frac{1}{2}F_{bb}(0,0,\Xi_{T,GLS})(\zeta_{2T})^2 \nonumber\\
    & + F_{ab}(0,0, \Xi_{T,GLS})(\zeta_{1T}\zeta_{2T}) .\label{eq:F_taylor_approx}
\end{align}
\noindent Again from Proposition \ref{prop:theta_u_indep} we observe that $\Upsilon_T$ is independent of $(\zeta_{1T}, \zeta_{2T}, \Xi_{T,GLS})$.  Therefore, 
\begin{align*}
    E\left [ F(\zeta_{1T}, \zeta_{2T}, \Xi_{T,GLS}) \right]= &  E\left [F(0,0,\Xi_{T,GLS}) \right]\\
    & + E[F_a(0,0, \Xi_{T,GLS}) (\zeta_{1T}) ]+ \frac{1}{2}E[F_{aa}(0,0, \Xi_{T,GLS})(\zeta_{1T})^2 ]\\
    & + E[F_b(0,0,\Xi_{T,GLS})\zeta_{2T} ]+ \frac{1}{2}E[F_{bb}(0,0,\Xi_{T,GLS})(\zeta_{2T})^2] \\
    & + E[F_{ab}(0,0, \Xi_{T,GLS})(\zeta_{1T}\zeta_{2T})]\\
    =&  E\left [F(0,0,\Xi_{T,GLS}) \right]\\
    & + E[F_a(0,0, \Xi_{T,GLS})]E[ \zeta_{1T} ]+ \frac{1}{2}E[F_{aa}(0,0, \Xi_{T,GLS})]E[ \zeta_{1T}^2 ]\\
    & + E[F_b(0,0,\Xi_{T,GLS})]E[ \zeta_{2T} ]+ \frac{1}{2}E[F_{bb}(0,0,\Xi_{T,GLS})]E[ \zeta_{2T}^2] \\
    & + E[F_{ab}(0,0, \Xi_{T,GLS})]E[ \zeta_{1T}\zeta_{2T}].\\
    \intertext{Observe that $ E[\zeta_{2T}] = O(T^{-1})$ and $E[ \zeta_{1T}^2 ]= O(T^{-1})$, hence }
    E\left [ F(\zeta_{1T}, \zeta_{2T}, \Xi_{T,GLS}) \right] = &  E\left [F(0,0,\Xi_{T,GLS}) \right] + E[F_a(0,0, \Xi_{T,GLS})]E[ \zeta_{1T} ]\\
    & + O \left(E[ \zeta_{1T}^2 ]\right)  + O \left( E[\zeta_{2T}] \right) + O \left(E[ \zeta_{2T}^2] \right) + O \left(E[ \zeta_{1T}\zeta_{2T}]\right) \\
    = & E\left [F(0,0,\Xi_{T,GLS}) \right] + E[F_a(0,0, \Xi_{T,GLS})]E[ \zeta_{1T} ] + O(T^{-1}).
\end{align*}
\noindent Let $f_e(x)$ be the probability density function (pdf) of $e_T$, then 
\begin{align*}
    E[F_a(0,0, \Xi_{T,GLS})\zeta_{1T} ]= & \int E [F_a(0,0, \Xi_{T,GLS})\zeta_{1T}|e_T = x]f_e(x) dx\\
    = &  \int E [F_a(0,0, x' \Omega^{1/2}_{GLS}\hat{\Omega}^{-1}_T\Omega^{1/2}_{GLS}x)] 2(\sqrt{T} R\Delta)' \hat{\Omega}^{-1} \Omega_{GLS}^{1/2}x f_e(x) dx.
\end{align*}
\noindent We note that $e_T$ is a standard normal vector, $\hat{\Omega}_T$ is an even function of the estimated errors, and $\Delta$ is an odd function of the estimated errors.  The product of an odd function and even functions results in an odd function, and the integral of an odd function on a symmetric interval about zero is zero.  Thus, $E[F_a(0,0, \Xi_{T,GLS})\zeta_{1T} ]= 0$, and  
\begin{align*}
    P \left(dF_{OLS} \leq z \right) = & E\left [F(0,0,\Xi_{T,GLS}) \right] + O(T^{-1})\\
    = & P(\Upsilon_T \Xi_{T, GLS} < z) + O(T^{-1})\\
    = & P(dF_{GLS} < z) + O(T^{-1}).
\end{align*}

\end{proof}

\begin{lemma} Assume that $b \rightarrow 0$  and $bT \rightarrow \infty$ as $T \rightarrow \infty$. Under Assumption 5 and using Remark \ref{remark:fsb_bias}, 
\begin{enumerate}
    \item[(a)] The Taylor series expansion of $\Xi_T^{-1}$ around $\Xi_T(\Omega)$ is 
    \begin{equation*}
        \Xi_T^{-1} = 1 + L + Q + o_p(b) 
    \end{equation*}
    \item[(b)] $E(L) = -(bT)^{-q}g_q w_q - c_1b - c_b + o\left[(bT)^{-q}\right]$
    \item[(c)] $E(Q) = -bc_2 (d-1) + o(b) $
    \item[(d)] $E(L^2) = 2c_2b + o(b) $
\end{enumerate}
where 
\begin{itemize}
    \item $\StackDem{L}{1 \times 1} = D vec(\hat{\Omega}_T - \Omega)$
    \item $\StackDem{Q}{1 \times 1} = \frac{1}{2}vec(\hat{\Omega}_T - \Omega)'(J_1 + J_2)vec(\hat{\Omega}_T - \Omega)$
    \item $\StackDem{D}{1 \times d^2} = \left ( \left[e_T' \Omega^{-1/2}  \right] \otimes \left[e_T' \Omega^{-1/2}  \right]\right ) $
    \item $\StackDem{J_1}{d^2\times d^2} =  \left [2 \Omega^{-1/2} (e_T e_T') \Omega^{-1/2} \right ] \otimes \left [ \Omega^{-1/2} ( e_T e_T') \Omega^{-1/2} \right] $\\
    \item $\StackDem{J_2}{d^2\times d^2} = - \left [\Omega^{-1/2} e_T e_T' \Omega^{-1/2} \otimes \Omega^{-1}\right] \mathbb{K}_{dd} \left ( \mathbb{I}_{d^2} + \mathbb{K}_{dd}\right).$
\end{itemize}
\label{lemma:taylor_series}
\end{lemma}


\begin{proof}[Proof of Lemma \ref{lemma:taylor_series}a]

\noindent First we observe that 
\begin{align*}
    \Xi_T (\Omega) =& e_T'[\Omega^{1/2}\Omega_T^{-1} \Omega^{1/2}]e_T\\
    =& e_T'e_T\\
    =& 1. 
\end{align*}
\noindent Therefore, $\Xi_T^{-1}(\Omega) = 1$.  We next compute the first derivative of $\Xi^{-1}_T$ with respect to
$\hat{\Omega}_T$ using the chain rule, and adopt the notation $\frac{\partial }{\partial vec \left  (\hat{\Omega }\right)  }= d$, 
\begin{align*}
     \frac{\partial }{\partial vec \left  (\hat{\Omega }\right)} \Xi^{-1}_T   =&  - \Xi^{-2}_T \frac{\partial }{\partial vec \left  (\hat{\Omega }\right)} \Xi_T \\
      d  \Xi^{-1}_T = & - \Xi^{-2}_T d \Xi_T\\
      = & - \Xi^{-2}_T d \left [ e_T'\Omega^{1/2} \hat{\Omega}_T^{-1} \Omega^{1/2} e_T \right ]\\
      = & - \Xi^{-2}_T  e_T'\Omega^{1/2} d \hat{\Omega}_T^{-1}  \Omega^{1/2} e_T. 
\end{align*}
\noindent We can use the matrix chain rule to observe that: $- d \hat{\Omega}^{-1} = \hat{\Omega}^{-1}_T (d \hat{\Omega}_T) \hat{\Omega}^{-1}_T$. Hence, 
\begin{align*}
    d  \Xi^{-1}_T = &  \Xi^{-2}_T e_T'\Omega^{1/2} \hat{\Omega}_T^{-1} (d \hat{\Omega}_T)\hat{\Omega}_T^{-1}\Omega^{1/2} e_T. 
\end{align*}
\noindent Furthermore, observe that for some matrix $\StackDem{Y}{d \times d}$ with components $[Y]_{i,j} = y_{ij}$, and vector $\StackDem{\mathbf{r}}{1 \times d} = [r_1, \dots, r_p]$ we have, $\StackDem{\left(\mathbf{r}\otimes \mathbf{r}\right) }{1 \times pp} \StackDem{vec(Y)}{dd \times 1} = \sum_i \sum_j r_i r_j y_{ij} = \mathbf{r} Y \mathbf{r}'$.  Let $\mathbf{r} = e'_T \Omega^{1/2} \hat{\Omega}^{-1} $ and recall $dvec \left(\hat{\Omega}_T \right) =1$ by definition,
\begin{align*}
   d \Xi^{-1}_T  = & - \Xi^{-2}_T \mathbf{r}  d \hat{\Omega}_T \mathbf{r}'\\
    = &  \Xi^{-2}_T \left [\mathbf{r} \otimes \mathbf{r}\right ]dvec \left(\hat{\Omega}_T \right)    \\
    = & \Xi^{-2}_T \left [\left( e'_T \Omega^{1/2} \hat{\Omega}^{-1} \right ) \otimes \left( e'_T \Omega^{1/2} \hat{\Omega}^{-1}\right )\right ]dvec \left(\hat{\Omega}_T \right)  \\
     = &  \Xi^{-2}_T \left [\left( e'_T \Omega^{1/2} \hat{\Omega}^{-1}_T \right ) \otimes \left( e'_T \Omega^{1/2}_T \hat{\Omega}^{-1}\right )\right ].
\end{align*}
\noindent Now we evaluate $d \Xi^{-1}_T$ at $\Omega$. Recall that $\Xi_T(\Omega) = 1$, thus $\Xi_T^{-2}(\Omega) = 1$, and 
\begin{align*}
    d \Xi_T^{-1} \left (\Omega \right )  = &  \left [\left( e'_T \Omega^{-1/2}  \right ) \otimes \left( e'_T \Omega^{-1/2} \right )\right ]\\
    := & D .
\end{align*}

\noindent We continue with the second derivative. Let $f(x)$ and $g(x) = g_1(x) + g_2(x)$ be arbitrary differentiable functions. Using the chain rule and product rule derive the following expression,  
\begin{align*}
    \frac{\partial}{\partial x} f^{-2}(x)g(x) = & (-2) f^{-3}(x) \left( \frac{\partial}{\partial x} f(x) \right ) g(x)   + f^{-2}(x)\left( \frac{\partial}{\partial x} g(x) \right ) \\
    = & (-2) f(x)^{-3} \left( \frac{\partial}{\partial x} f(x) \right ) g(x)   + f^{-2}(x)\left( \frac{\partial}{\partial x} g_1(x) \right ) + f^{-2}(x)\left( \frac{\partial}{\partial x} g_2(x) \right ) \\
    = & D_1 + D_2 + D_3. 
\end{align*}
\noindent The above notation is a general expression, for our case we have the following, 
\begin{equation*}
    d \frac{\partial \Xi^{-1}}{\partial \left [ vec (\hat{\Omega})\right ]'} = D_1 + D_2 + D_3 
\end{equation*}
\noindent where $f(x) = \Xi_T$ and $g(x) = \left [\left( e'_T \Omega^{-1/2}  \hat{\Omega}_T^{-1}\right ) \otimes \left( e'_T \Omega^{-1/2} \hat{\Omega}_T^{-1}\right )\right ]$
\begin{align*}
    D_1 = & -2 \Xi^{-3}_T  d \Xi_T\left [\left( e'_T \Omega^{1/2} \hat{\Omega}^{-1}_T \right ) \otimes \left( e'_T \Omega^{1/2} \hat{\Omega}^{-1}_T \right )\right ]\\
    D_2 = & - \Xi^{-2}_T \left [\left( e'_T \Omega^{1/2} (d  \hat{\Omega}_T^{-1}) \right ) \otimes \left( e'_T \Omega^{1/2} \hat{\Omega}^{-1} \right )\right ]\\
    D_3 = & - \Xi^{-2}_T \left [\left( e'_T \Omega^{1/2} \hat{\Omega}^{-1}_T \right ) \otimes \left( e'_T \Omega^{1/2} (d \hat{\Omega}^{-1}_T) \right )\right ]\\
    = & D_2 \mathbb{K}_{dd}.
\end{align*}
\noindent We can find alternative representations for these terms.  For example, observe that 
\begin{align*}
    d \Xi_T =& \left [\left( e'_T \Omega^{1/2} \hat{\Omega}^{-1}_T \right ) \otimes \left( e'_T \Omega^{1/2} \hat{\Omega}^{-1}_T \right )\right ]dvec \left(\hat{\Omega}_T \right)\\
    = & dvec \left(\hat{\Omega}_T \right )'\left [\left( e'_T \Omega^{1/2} \hat{\Omega}^{-1}_T \right ) \otimes \left( e'_T \Omega^{1/2} \hat{\Omega}^{-1}_T \right )\right ]',
\end{align*}
\noindent a relationship observed during the derivation for the first derivative of $\Xi_T^{-1}$. We can substitute this term inside of $D_1$, 
\begin{align*}
        D_1 =  & -2 \Xi_T^{-3} dvec \left(\hat{\Omega}_T \right )'   \left [\left( e'_T \Omega^{1/2} \hat{\Omega}^{-1}_T \right ) \otimes \left( e'_T \Omega^{1/2} \hat{\Omega}^{-1}_T \right )\right ]' \\
         & \times \left [\left( e'_T \Omega^{1/2} \hat{\Omega}^{-1}_T \right ) \otimes \left( e'_T \Omega^{1/2} \hat{\Omega}^{-1}_T \right )\right ].
\end{align*}
\noindent We can also find an alternative representation for $D_2$ and $D_3$.  In this alternative representation we can use a result from earlier, that $-d \hat{\Omega}^{-1}_T = \hat{\Omega}_T^{-1} (d\hat{\Omega}_T ) \hat{\Omega}_T^{-1}$, and observe 
\begin{align*}
     D_2 = & - \Xi^{-2}_T  \left [\left( e'_T \Omega^{1/2} d  \hat{\Omega}_T^{-1} \hat{\Omega}_T  \hat{\Omega}_T^{-1}\right ) \otimes \left( e'_T \Omega^{1/2} \hat{\Omega}^{-1}_T \right )\right ]\\
     = & \Xi^{-2}_T  \left [\left( e'_T \Omega^{1/2}  \hat{\Omega}_T^{-1} (d \hat{\Omega}_T)\hat{\Omega}_T^{-1}\right ) \otimes \left( e'_T \Omega^{1/2} \hat{\Omega}^{-1}_T \right )\right ]\\
    = &  \Xi^{-2}_T  \left [\left(\hat{\Omega}_T^{-1} (d\hat{\Omega}_T )  \hat{\Omega}_T^{-1}  \Omega^{1/2} e_T \right )' \otimes  \left( e'_T \Omega^{1/2} \hat{\Omega}^{-1}_T\right )\right ]\\ 
    = &  \Xi^{-2}_T  \left [\left(\hat{\Omega}_T^{-1} (d\hat{\Omega}_T) \left [e'_T  \Omega^{1/2}  \hat{\Omega}_T^{-1}   \right ]' \right )'\otimes  \left( e'_T \Omega^{1/2} \hat{\Omega}^{-1}_T \right )\right ]\\
    = & \Xi^{-2}_T  \left [\left(ZY\mathbf{r}' \right )'\otimes  \left( e'_T \Omega^{1/2} \hat{\Omega}^{-1}_T \right )\right ]
\end{align*}
\noindent where $Z =\hat{\Omega}_T^{-1}$, $Y = d \Omega_T$, and $r =\left [e'_T  \Omega^{1/2}  \hat{\Omega}_T^{-1} \right ]$. Notice vector $ZY\mathbf{r}' = vec(ZY\mathbf{r}')$ and $vec(ZYr') = (r \otimes Z) vec(Y)$, which gives us
\begin{align*}
     D_2 = & \Xi^{-2}_T  \left [\left(vec(ZY\mathbf{r}') \right )'\otimes  \left( e'_T \Omega^{1/2} \hat{\Omega}^{-1} \right )\right ]\\
     = & \Xi^{-2}_T  \left [\left((r \otimes Z) vec(Y) \right )'\otimes  \left( e'_T \Omega^{1/2} \hat{\Omega}^{-1}_T \right )\right ].
\end{align*}
\noindent Substituting values back in we obtain, 
\begin{align*}
    D_2 = & \Xi^{-2}_T  \left [ \left (\left(\left [e'_T  \Omega^{1/2}  \hat{\Omega}_T^{-1} \right ] \otimes  \hat{\Omega}_T^{-1}\right )  d vec \left(\hat{\Omega}_T \right ) \right )'\otimes  \left( e'_T \Omega^{1/2} \hat{\Omega}^{-1}_T \right ) \right] \\
    =  & \Xi^{-2}_T d vec \left(\hat{\Omega}_T \right )' \left [ \left(\left [e'_T  \Omega^{1/2}  \hat{\Omega}_T^{-1} \right ] \otimes  \hat{\Omega}_T^{-1}\right )  '\otimes  \left( e'_T \Omega^{1/2} \hat{\Omega}^{-1}_T \right ) \right ] \\
    = & \Xi^{-2}_T d vec \left(\hat{\Omega}_T \right )' \left [\left( e'_T  \Omega^{1/2}  \hat{\Omega}_T^{-1} \right )' \otimes  \hat{\Omega}_T^{-1}\otimes  \left( e'_T \Omega^{1/2} \hat{\Omega}^{-1}_T \right )  \right ].
\end{align*}
\noindent Now we can calculate the second derivative of $\Xi_T^{-1}$ at $\Omega$.  Recall $dvec \left(\hat{\Omega}\right) =1$, $\Xi_T \left( \Omega\right ) =1$, and $(A\otimes B)(C \otimes D) = (AC)\otimes (BD)$. Therefore,  
\begin{align*}
    \frac{\partial^2 \Xi_T \left( \Omega\right)}{\partial vec \left(\hat{\Omega}\right )\partial vec \left(\hat{\Omega}\right ) }  = & 2 (1) \left ( \left [e_T' \Omega^{-1/2} \right ]' \otimes  \left [e_T' \Omega^{-1/2} \right ]'  \right ) \left( \left [e_T' \Omega^{-1/2} \right ] \otimes  \left [e_T' \Omega^{-1/2} \right ]\right ) \\
    & -\left(  \left [e_T' \Omega^{-1/2} \right ]' \otimes \Omega^{-1} \otimes  \left [e_T' \Omega^{-1/2} \right ]\right ) \left(\mathbb{I}_{d^2} + \mathbb{K}_{dd} \right ) \\
    = & 2 \left [\Omega^{-1/2} e_T e'_T \Omega^{-1/2} \right ] \otimes \left [\Omega^{-1/2} e_T e'_T \Omega^{-1/2} \right ]\\
     & - \left [e_T' \Omega^{-1/2} \right ]' \otimes \Omega^{-1} \otimes  \left [e_T' \Omega^{-1/2} \right ]\left(\mathbb{I}_{d^2} + \mathbb{K}_{dd} \right ) \\
     = & 2 \left [\Omega^{-1/2} e_T e'_T \Omega^{-1/2} \right ] \otimes \left [\Omega^{-1/2} e_T e'_T \Omega^{-1/2} \right ]\\
     & -\left[\Omega^{-1/2} e_T e'_T \Omega^{-1/2} \otimes \Omega^{-1} \right ]\mathbb{K}_{dd}\left(\mathbb{I}_{d^2} + \mathbb{K}_{dd} \right )\\
     := & J_1 + J_2.
\end{align*}

\end{proof}


\begin{proof}[Proof of Lemma \ref{lemma:taylor_series}b]

\noindent Recall Proposition \ref{prop:theta_u_indep}, which implies $e_T \perp \hat{\Omega}_T$. Therefore, under Theorem 1 we have, 
\begin{align*}
    E (L) = & E [D vec (\hat{\Omega}_T - \Omega)]\\
    = & E e'_T \Omega^{-1/2}(\hat{\Omega}_T - \Omega)\Omega^{-1/2} e_T\\
    = & -(bT)^{-q}E e'_T \Omega^{-1/2}(g_qh_q)\Omega^{-1/2} e_T - (c_1 + c_b) Ee'_T \Omega^{-1/2}\Omega\Omega^{-1/2} e_T + o\left ((bT)^{-q} \right )\\
    = & -(bT)^{-q}E e'_T \Omega^{-1/2}(g_qh_q)\Omega^{-1/2} e_T - (c_1 + c_b) + o\left ((bT)^{-q} \right ).
\end{align*}
\noindent For a symmetric matrix $\StackDem{A}{d \times d}$ and row vector $\StackDem{\mathbf{r}}{d \times 1}$: $\mathbf{r}' A \mathbf{r} = \sum_j r_j \left (\sum_i r_i a_{ij} \right) $ and $[A \mathbf{r} \mathbf{r}']_{ij} = r_j\left(\sum_i a_{ij} r_i \right )$.  Therefore $trace(A \mathbf{r} \mathbf{r}') = \mathbf{r}' A \mathbf{r}=\sum_j r_j \left(\sum_i a_{ij} r_i \right) $.  We further observer that $E[e_Te'_T] = \mathbb{I}_d E \left(\frac{1}{||\Omega_{GLS}^{-1/2} \sqrt{T}(R_0 \hat{\theta}_{GLS} - r_0)||}\right) = \frac{1}{d}$.  With these properties we obtain, 
\begin{align*}
    = & -(bT)^{-q}g_q E \left[trace\left(h_q \Omega^{-1} e_T e'_T\right) \right ]- c_1 - c_b + o\left ((bT)^{-q} \right )\\
    = & -(bT)^{-q}\frac{g_q}{d} trace\left(h_q \Omega^{-1}\right) -c_1 - c_b + o\left ((bT)^{-q} \right )\\
    = &-(bT)^{-q}g_q w_q -c_1 - c_b + o\left ((bT)^{-q} \right ).
\end{align*}

\end{proof}

\begin{proof}[Proof of Lemma \ref{lemma:taylor_series}c]

\noindent We start by finding an expression for $$E(Q) = \frac{1}{2} E \left [vec \left( \hat{\Omega}_T - \Omega\right ) \left(J_1 + J_2 \right ) \left( \hat{\Omega}_T - \Omega\right) \right ]. $$  We first study just the $J_1$ term. Rearranging terms, 
\begin{align*}
    \frac{1}{2} E \left [vec \left(\hat{\Omega}_T - \Omega \right )' J_1 vec \left(\hat{\Omega}_T - \Omega \right )  \right ] = &  \frac{2}{2}E \left [vec \left(\hat{\Omega}_T - \Omega \right )' \left [\Omega^{-1/2} e_T e'_T \Omega^{-1/2} \right ]\right . \\
    &  \left .  \otimes \left [\Omega^{-1/2} e_T e'_T \Omega^{-1/2} \right ] vec \left(\hat{\Omega}_T - \Omega \right )  \right ]\\
    = &  E \left [vec \left(\hat{\Omega} - \Omega \right )' 
 vec \left \{\Omega^{-1/2} e_T e'_T \Omega^{-1/2} \right ]\right . \\
    &  \left .\left .   \times  \left(\hat{\Omega} - \Omega \right )\Omega^{-1/2}e_T e'_T \Omega^{-1/2}  \right \} \right ] \\
    = & E \left [ \left (e_T' \Omega^{-1/2} \left (\hat{\Omega} - \Omega \right ) \Omega^{-1/2} e_T \right )^2 \right ]\\
        = & E [(X' A X)^2]
\end{align*}
\noindent where $A = \Omega^{-1/2} \left (\hat{\Omega} - \Omega \right ) \Omega^{-1/2}$, $[A]_{ij} = A_{ij}$,$ X = e_T$. Using the tower property of expectations and Proposition \ref{prop:theta_u_indep} we can rewrite this as the following,
\begin{align*}
    E [(X' A X)^2] =&  E \left [\sum_{i,j}\sum_{m,l} A_{ij}A_{lm}X_iX_j X_l X_m \right ]\\
    = & E \left [\sum_{i,j}\sum_{m,l} A_{ij}A_{lm}E[X_iX_j X_l X_m|A] \right ]\\
    = & \sum_{i,j}\sum_{m,l} E[A_{ij}A_{lm}]E[X_iX_j X_l X_m].    
\end{align*}
\noindent Without loss of generality, if $X_i$ is different from $(X_j, X_l, X_m)$, then $E[X_iX_j X_l X_m] = 0$ \citep{khokhlov2006uniform, sun2011robust}. Thus, we only care about four cases: $(i=j=l=m)$, $(i=j, l=m)$, $(i=l, m=j)$, and $(i=m, j=l)$.  Furthermore, $E[X_i^4] = \frac{3}{d(d+2)}$ and $E[X_i^2 X_j^2] = \frac{1}{d(d+2)}$. We break up the expectation into these four cases, 

\begin{align*}
    \sum_{i,j}\sum_{m,l} A_{ij}A_{lm}E[X_iX_j X_l X_m] = & \sum_{i=j=l=m} A_{ii}^2 E(X_i^4) + \sum_{i=j} \sum_{l=m} A_{ii} A_{jj} E[X_i^2 X_m^2] \\
    & + \sum_{i=l}\sum_{j=m} A_{ij}^2 E[X_i^2 X_j^2]  + \sum_{j=l}\sum_{j=l}A_{ij}^2 E[X_i^2 X_j^2] \\
    = & \sum_{i=j=l=m} A_{ii}^2 E(X_i^4) + \sum_{i=j} \sum_{l=m} A_{ii} A_{jj} E[X_i^2 X_m^2] \\
    & + 2\sum_{i=l}\sum_{j=m} A_{ij}^2 E[X_i^2 X_j^2]\\
    = & \sum_{i=j=l=m} A_{ii}^2 \left(\frac{3}{d(d+2)} \right)  + \sum_{i=j} \sum_{l=m} A_{ii} A_{jj} \left (\frac{1}{d(d+2)}\right)\\
    & + 2\sum_{i=l} \sum_{j=m} A_{ij}^2 \left(\frac{1}{d(d+2)}\right ).
\end{align*}
\noindent We still need the expected value of the $A$ terms. Notice that the off-diagonals of $A$ squared are 0, i.e. $E[A_{ij}^2] = 0$ for $i \neq j$. Therefore, 
\begin{align*}
    \sum_{i,j}\sum_{m,l}E \left [  A_{ij}A_{lm}\right]E[X_iX_j X_l X_m]= & \sum_{i=j=l=m}E [ A_{ii}^2] \left(\frac{3}{d(d+2)} \right)   \\
    & + \sum_{i=j} \sum_{l=m} E[A_{ii} A_{jj}]\left (\frac{1}{d(d+2)}\right)\\
    =  &\left(\frac{3}{d(d+2)} \right) E \left [ trace \left (AA\right ) \right ] \\
    & + \left (\frac{1}{d(d+2)}\right)E \left [ trace^2 \left (A\right ) - trace(AA)\right ]\\
    = & \left (\frac{1}{d(d+2)}\right) 2 \left[ E \left [ trace(AA)\right] + E \left[ trace^2(A)\right ]\right ].
\end{align*}

\noindent Next we find $E[trace(AA)]$. We use a trace property to rearrange,  

\begin{align*}
    E[trace(AA)] = & E \left [ trace \left (vec'(A) vec(A)\right)  \right ]\\
    = & E \left [ trace \left (vec'(\Omega^{-1/2}(\hat{\Omega}_T - \Omega) \Omega^{-1/2}  ) vec(\Omega^{-1/2}(\hat{\Omega}_T - \Omega) \Omega^{-1/2}  )\right)  \right ]\\
    = & E \left [ trace \left (vec'(\hat{\Omega}_T - \Omega) (\Omega^{-1/2} \otimes \Omega^{-1/2})(\Omega^{-1/2} \otimes \Omega^{-1/2}) vec(\hat{\Omega}_T - \Omega) \right)  \right ]\\
    = & E \left [ trace \left ( (\Omega^{-1/2} \otimes \Omega^{-1/2})vec(\hat{\Omega}_T - \Omega)vec'(\hat{\Omega}_T - \Omega)(\Omega^{-1/2} \otimes \Omega^{-1/2}) \right)  \right ].
\end{align*}
\noindent To further simplify we use Theorem 1, a trace property of symmetric matrices, and a multiplication property of matrices, 
\begin{align*}
    E[trace(AA)]= & trace \left ( (\Omega^{-1/2} \otimes \Omega^{-1/2})Var \left( vec(\hat{\Omega}_T \right) (\Omega^{-1/2} \otimes \Omega^{-1/2})\right ) \\
    =&  b c_2 trace \left ( (\Omega^{-1/2} \otimes \Omega^{-1/2}) \left( \mathbb{I}_{d^2} + \mathbb{K}_{dd} \right)\left (\Omega \otimes \Omega \right)  (\Omega^{-1/2} \otimes \Omega^{-1/2})\right ) +o(b)  \\
    = &  b c_2 trace \left ( (\Omega^{-1/2} \otimes \Omega^{-1/2})\left (\Omega \otimes \Omega \right)  (\Omega^{-1/2} \otimes \Omega^{-1/2}) \left( \mathbb{I}_{d^2} + \mathbb{K}_{dd} \right)\right ) +o(b)\\ 
    = &  b c_2 trace \left ( (\Omega^{0} \otimes \Omega^{0}) \left( \mathbb{I}_{d^2} + \mathbb{K}_{dd} \right)\right ) +o(b)\\ 
    = &  b c_2 trace \left (   \mathbb{I}_{d^2} + \mathbb{K}_{dd} \right) +o(b)\\ 
    = &  b c_2\left [ trace \left (   \mathbb{I}_{d^2} \right) + trace \left (  \mathbb{K}_{dd} \right)\right] +o(b)\\ 
    = &  b c_2\left [ d^2 + d \right] +o(b).    
\end{align*}
\noindent Next we derive $E \left[trace^2(A) \right ]$. To further simplify we use a trace property for symmetric matrices, and observe that $(\sum_i A_{ii})^2 =\sum_i \sum_j A_{ii} Ajj$, 
\begin{align*}
    E \left[trace^2(A) \right ] = & E \left[trace^2\left (\Omega^{-1/2} (\hat{\Omega}_T - \Omega) \Omega^{-1/2} ) \right)  \right ] \\
    = & E \left[trace^2\left ( (\hat{\Omega}_T - \Omega) \Omega^{-1} ) \right)  \right ] \\
    = & E \left[trace\left ( (\hat{\Omega}_T - \Omega) \Omega^{-1} ) \right)trace\left ( (\hat{\Omega} - \Omega) \Omega^{-1} ) \right)  \right ] \\
     = & E \left[vec'\left ( (\hat{\Omega}_T - \Omega) \Omega^{-1} ) \right)vec( \left ( (\hat{\Omega}_T - \Omega) \Omega^{-1} ) \right)  \right ]. 
\end{align*}
\noindent Let $[\hat{\Omega}_T - \Omega]_{ij} = \phi_{ij}$, and $[\Omega^{-1}]_{ij}= \omega_{ij}$.  For $i \neq j$, $E(\phi_{ij}^2) = 0$, and for $(i, j) = (l, k) $ we have that $E(\phi_{ij}\phi_{lk}) = 0$.  Hence $vec'\left ( (\hat{\Omega}_T - \Omega) \Omega^{-1} ) \right)vec \left ( (\hat{\Omega}_T - \Omega) \Omega^{-1} ) \right) = \sum_l (\phi\omega)_l^2 = \sum_l \phi_{ll}^2\omega_{ll}^2$.  We also observe that
\begin{align*}
    vec'(\Omega^{-1}) vec(\hat{\Omega}_T-\Omega) vec'(\hat{\Omega}_T-\Omega) vec(\Omega^{-1}) &  = \left(\sum_l (\omega\phi)_l \right) \left(\sum_h (\omega\phi)_h \right) \\
    = & \sum_l \sum_h (\omega\phi)_l \sum_l (\omega\phi)_h \\
    = & \sum_l  (\omega\phi)_l^2  \\
    = & \sum_l  \omega_{ll}^2 \phi_{ll}^2.
\end{align*}
\noindent Therefore 
\begin{equation*}
vec'\left ( (\hat{\Omega}_T - \Omega) \Omega^{-1} ) \right)vec( \left ( (\hat{\Omega}_T - \Omega) \Omega^{-1} ) \right)= vec'(\Omega^{-1}) vec(\hat{\Omega}_T-\Omega) vec'(\hat{\Omega}_T-\Omega) vec(\Omega^{-1}), 
\end{equation*}
\noindent and we have the following,
\begin{align*}
     E \left[trace^2(A) \right ] =& E \left[vec'(\Omega^{-1}) vec(\hat{\Omega}_T-\Omega) vec'(\hat{\Omega}_T-\Omega) vec(\Omega^{-1}) \right ].
\end{align*}
\noindent Furthermore by Theorem \ref{thm:bias_var}, 
\begin{align*}
     E \left[trace^2(A) \right ] =& vec'(\Omega^{-1}) Var(vec(\hat{\Omega}_T)) vec(\Omega^{-1})\\
     = & b c_2 vec'(\Omega^{-1}) \left(\Omega \otimes \Omega \right )\left(\mathbb{I}_{d^2} + \mathbb{K}_{dd} \right)  vec(\Omega^{-1})+ o(b) \\
     = & 2 b c_2 vec'(\Omega^{-1}) \left(\Omega \otimes \Omega \right )vec(\Omega^{-1})+ o(b) \\
     = & 2 b c_2 vec'(\Omega^{-1})vec(\Omega\Omega^{-1}\Omega)+ o(b) \\
    = & 2 b c_2 vec'(\Omega^{-1})vec(\Omega)+ o(b) \\
    = & 2b c_2 d + o(b). 
\end{align*}
\noindent We now have what we need for $\frac{1}{2} E \left[vec' \left(\hat{\Omega}_T - \Omega \right ) J_1 vec \left(\hat{\Omega}_T - \Omega \right )\right ] $,
\begin{align*}
    \frac{1}{2} E \left[vec' \left(\hat{\Omega}_T - \Omega \right ) J_1 vec \left(\hat{\Omega}_T - \Omega \right )\right ] = & \frac{1}{d(d+2)} \left[2 E \left(trace(AA)\right) + E \left(trace^2(A) \right)  \right ]\\
    =  & \frac{1}{d(d+2)} \left[2 (bc_2)(d^2+d) + 2bc_2d \right ]+ o(b) \\
    = & \frac{2b c_2d}{d(d+2)} \left[(d+1) + 1 \right ]+ o(b) \\
    = & \frac{2b c_2}{(d+2)} \left[d+2\right ]+ o(b) \\
    = & 2b c_2+ o(b). 
\end{align*}
\noindent Next we derive $\frac{1}{2} E \left[vec' \left(\hat{\Omega}_T - \Omega \right ) J_2 vec \left(\hat{\Omega}_T - \Omega \right )\right ]$. Recall for a symmetric commutation matrix $\mathbb{K}_{dd}\mathbb{K}'_{dd} = \mathbb{I}_{dd}$.  Utilizing this relationship, and rearranging terms we obtain, 
\begin{align*}
    vec' \left(\hat{\Omega}_T - \Omega \right ) J_2 vec \left(\hat{\Omega}_T - \Omega \right ) = & -vec' \left(\hat{\Omega}_T - \Omega \right )\left [\Omega^{-1/2} e_T e_T' \Omega^{-1/2} \otimes \Omega^{-1}\right] \mathbb{K}_{dd} \\
    & \times \left ( \mathbb{I}_{d^2} + \mathbb{K}_{dd}\right) vec \left(\hat{\Omega}_T - \Omega \right ) \\
    = & -vec' \left(\hat{\Omega}_T - \Omega \right )\left [\Omega^{-1/2} e_T e_T' \Omega^{-1/2} \otimes \Omega^{-1}\right] \mathbb{K}_{dd}vec \left(\hat{\Omega}_T - \Omega \right )  \\
    & -vec' \left(\hat{\Omega}_T - \Omega \right )\left [\Omega^{-1/2} e_T e_T' \Omega^{-1/2} \otimes \Omega^{-1}\right] \mathbb{I}_{dd}vec \left(\hat{\Omega}_T - \Omega \right ) \\
    = & -2vec' \left(\hat{\Omega}_T - \Omega \right )\left [\Omega^{-1/2} e_T e_T' \Omega^{-1/2} \otimes \Omega^{-1}\right] vec \left(\hat{\Omega}_T - \Omega \right )  \\
     = & -2vec' \left(\hat{\Omega}_T - \Omega \right )vec \left( \Omega^{-1}( \hat{\Omega}_T - \Omega )\Omega^{-1/2} e_T e_T' \Omega^{-1/2}\right )  \\
     = & -2vec' \left(\hat{\Omega}_T - \Omega \right )vec \left( \Omega^{-1/2}A e_T e_T' \Omega^{-1/2}\right )  \\
     = & -2trace\left(vec \left( \Omega^{-1/2}A e_T e_T' \Omega^{-1/2}\right )vec' \left(\hat{\Omega}_T - \Omega \right ) \right)\\
     = & -2trace\left( \Omega^{-1/2}A e_T e_T' \Omega^{-1/2}(\hat{\Omega}_T - \Omega) \right).
\end{align*}
\noindent We can permute the arrangements of symmetric matrices within the trace operator and simplify, 
\begin{align*}
      = & -2trace\left( A e_T e_T' \Omega^{-1/2}(\hat{\Omega}_T - \Omega )\Omega^{-1/2}\right)\\
      = & -2trace\left( A X X'A \right)\\
      = & X' A A X.
\end{align*}
\noindent Next we find the expected value of $X' A A X$. By the tower property and Proposition \ref{prop:theta_u_indep},
\begin{align*}
    E[X' A A X] = & E \left [ \sum_{i, j, l} A_{il}A_{lj} X_i X_j \right ]\\
    = & E \left [ \sum_{i, j, l} A_{il}A_{lj} E[X_i X_j|A] \right ]\\
    = & \sum_{i, j, l}  E \left [A_{il}A_{lj} \right ] E\left[X_i X_j\right ]
\end{align*}
\noindent Recall: $E[X_i X_j] = 0 $ if $i \neq j$, and $A$ is symmetric. Therefore, 
\begin{align*}
     E[X' A A X] = & \sum_{l}\sum_{i=j}E \left [ A_{il}A_{li} \right ] E \left [X_i^2 \right ]\\   
     = & \sum_{l}\sum_{i=j}E \left [  A_{il}A_{li} \right ]\frac{1}{d} \\   
     = & \sum_{l}\sum_{i=j}E \left [  A_{il}^2 \right ]\frac{1}{d} \\   
     = & \sum_{i=j=l}E \left [  A_{ii}^2 \right ]\frac{1}{d} \\   
     = & E \left [ trace(AA) \right ]\frac{1}{d} \\  
     = & \frac{1}{d}  \left [ c_2 b (d^2+d ) \right ]\\  
    = &  c_2 b (d+1 ). 
\end{align*}
\noindent Hence, 
\begin{align*}
    \frac{1}{2} E \left[vec' \left(\hat{\Omega}_T - \Omega \right ) J_2 vec \left(\hat{\Omega}_T - \Omega \right )\right ] = &  \frac{1}{2} (-2) c_2 b (d+1 )\\
    = & -c_2 b (d+1). 
\end{align*}
\noindent We now have the components needed to obtain the expected value of $Q$, 
\begin{align*}
    E[Q] = & \frac{1}{2} E \left[vec' \left(\hat{\Omega}_T - \Omega \right ) (J_1 + J_2) vec \left(\hat{\Omega}_T - \Omega \right )\right ]\\
    = & 2b c_2-c_2 b (d+1)+ o(b)  \\
    = & c_2 b [2 - d -1]+ o(b) \\
    = & -c_2[d-1]b+ o(b).  
\end{align*}
\end{proof}
\begin{proof}[Proof of Lemma \ref{lemma:taylor_series}d.] 

\noindent We next derive the expected value of $L^2$, 
\begin{align*}
    E[L^2] = & E \left[Dvec \left(\hat{\Omega}_T - \Omega \right)  D vec \left (\hat{\Omega}_T - \Omega \right) \right ]\\
    = & E \left[Dvec \left(\hat{\Omega}_T - \Omega \right)   vec' \left (\hat{\Omega}_T - \Omega \right)D' \right ]\\
    = & E \left[\left( e'_T \Omega^{-1/2} \otimes e'_T \Omega^{-1/2} \right) vec \left(\hat{\Omega}_T - \Omega \right)   vec' \left (\hat{\Omega}_T - \Omega \right)\left( e'_T \Omega^{-1/2} \otimes e'_T \Omega^{-1/2} \right)' \right ]\\
    = & E \left[\left( e'_T \Omega^{-1/2} \otimes e'_T \Omega^{-1/2} \right) vec \left(\hat{\Omega}_T - \Omega \right)   vec' \left (\hat{\Omega}_T - \Omega \right)\left( e'_T \Omega^{-1/2} \otimes e'_T \Omega^{-1/2} \right)' \right ].
\end{align*}
\noindent By Proposition \ref{prop:theta_u_indep} and Theorem \ref{thm:bias_var} we obtain, 
\begin{align*}
      E[L^2]= & c_2 b E \left[\left( e'_T \Omega^{-1/2} \otimes e'_T \Omega^{-1/2} \right) \left(\mathbb{I}_{d^2} + \mathbb{K}_{dd} \right)\left(\Omega \otimes \Omega \right)  \left( e'_T \Omega^{-1/2} \otimes e'_T \Omega^{-1/2} \right)' \right ]+ o(b) \\
     = & 2c_2 b E \left[\left( e'_T \Omega^{-1/2} \otimes e'_T \Omega^{-1/2} \right) \left(\Omega \otimes \Omega \right)  \left( e'_T \Omega^{-1/2} \otimes e'_T \Omega^{-1/2} \right)' \right ]+ o(b) \\
     = & c_2 b E \left[\left( e'_T \Omega^{-1/2} \otimes e'_T \Omega^{-1/2} \right)  \left(  \Omega^{1/2}e_T \otimes \Omega^{1/2}e_T \right) \right ]+ o(b) \\  
     = & 2c_2 b E \left[\left( e'_T e_T\otimes e'_T e_T \right) \right ]+ o(b) \\ 
     = & 2c_2 b E \left[\left( e'_T e_T \right)^2 \right ]+ o(b).
\end{align*}
\noindent Let $X = e_T$ as before, then  
\begin{align*}
    E \left[\left( e'_T e_T \right)^2 \right ] = & E \left[\left(\sum_{i} X_i^2\right )^2 \right ]\\
    = & E \left[\left(\sum_{i}\sum_j X_i^2 X_j^2\right ) \right ]\\
    = & E \left[\left(\sum_{i=j}X_i^4 + \sum_{i \neq j} X_i^2 X_j^2\right ) \right ]\\
    = & d \left (\frac{3}{d(d+2} \right ) + (d^2 - d) \left( \frac{1}{d(d+2)}\right )  \\
    = &\frac{1}{d+2} \left(3 + (d-1)  \right ) \\
    = & 1.
\end{align*}
\noindent Therefore, $E(L^2) =  2c_2b+ o(b) $ as desired. 

\end{proof}



\begin{proof}[Proof of Theorem \ref{thm:null_alt_distr}b]

\noindent We start using similar steps as in \cite{sun2014let}. We start first consider the case that $\delta = 0$. Using Lemma \ref{lemma:chi_sq_xi} and  Lemma \ref{lemma:taylor_series},
\begin{align*}
    P(d F_{OLS} \leq z) =&  E \left[G_d\left( z \Xi_T^{-1}\right ) \right]\\
    = & E\left \{G_d\left [z (1 + L + Q + o_p(b)) \right] \right \}.\\
    \intertext{We next take the Taylor Series expansion of $G_d\left (\cdot\right)$ around $G_d(z)$. We also use the results of Lemma \ref{lemma:taylor_series}, and further observe that the derivatives of $G'_d(x)$ and $G''_d(x)$ are bounded. We use a similar argument from Lemma \ref{lemma:chi_sq_xi} to simplify,}
     P(d F_{OLS} \leq z) = & E\left \{ G_d(z) + \frac{G_d(z)}{1!}z( L + Q) + \frac{G''_d(z)}{2}z^2(L+Q)^2\right \} + o(b)\\
    = &  G_d(z) + \frac{G_d(z)}{1!}zE(L + Q) + \frac{G''_d(z)}{2}z^2E\left ((L+Q)^2\right ) +  o(b)\\   
    = &  G_d(z) + \frac{G_d(z)}{1!}zE(L + Q) + \frac{G''_d(z)}{2}z^2E\left (L^2\right ) +  o(b)\\
= &  G_d(z) + \frac{G_d(z)}{1!}z\left(-(bT)^{q}g_qw_q-c_1b - c_b-bc_2(d-1) \right ) \\
    & + \frac{G''_d(z)}{2}z^2\left (2c_2b\right ) + o(b) + o\left((bT)^{-q}\right) \\   
    = &  G_d(z) + G^{''}_d(z) z^2b c_2 - G'_d(z) z[c_b + c_1b+  c_2 (d-1)b]  \\
    & - (bT)^{-q}G'_d(z) g_q w_q + o(b) + o \left( (bT)^{-q}\right). 
\end{align*}

\noindent Note that, as observed by \cite{sun2014let}, the remainder term above is absorbed into the other error terms. Next we consider the case when $\delta \neq 0$. Define
\begin{align*}
     \StackDem{e_\delta}{d \times 1} :=& \frac{\Omega^{-1/2}_{GLS} \left [ R \sqrt{T} (\hat{\theta}_{GLS} - \theta_0)\right]  + \delta}{||\Omega^{-1/2}_{GLS} \left [ R \sqrt{T} (\hat{\theta}_{GLS} - \theta_0)\right]  + \delta||}\\
    \StackDem{\Upsilon_\delta}{1\times 1} := &||\Omega_{GLS} \left [ R \sqrt{T} (\hat{\theta}_{GLS} - \theta_0)\right]  + \delta||^2 \\
    \StackDem{\Xi_\delta}{1\times 1} := & e_\delta'\Omega_{GLS}^{1/2}\hat{\Omega}_T^{-1} \Omega_{GLS}^{1/2} e_\delta.
\end{align*}
\noindent Using similar arguments as in Lemma \ref{lemma:chi_sq_xi}, under $H_A$ with non-centrality parameter $\delta^2$
\begin{equation*}
    d F_{GLS} = \Upsilon_\delta  \Xi_\delta
\end{equation*}
\noindent which implies 
\begin{align*}
    P_\delta \left(d F_{OLS} \leq z \right ) = & P_\delta \left(\Upsilon_\delta  \Xi_\delta \leq z \right )  + O(T^{-1})\\
    = & E G_{d, \delta^2}\left( z \Xi_T^{-1}\right ) + O(T^{-1})
\end{align*}
\noindent for some $z \in \mathbb{R}$. Use similar steps as in Lemma \ref{lemma:taylor_series} and the case when $\delta = 0$, 
\begin{align*}
    P_\delta(d F_{OLS} \leq z) = & G_{d, \delta^2}\left (z\right) + G^{''}_{d, \delta^2}\left (z\right)b \left (z\right) ^2 c_2\\
    & - G'_{d, \delta^2}(z) z[c_{b,T} + c_1 b+ c_2 (d-1)b]  \\
    & - (bT)^{-q}G'_{d, \delta^2}\left (z\right)zg_q w_q + o(b) + o \left( (bT)^{-q}\right). 
\end{align*}
\end{proof}

\begin{proof}[Proof Corollary \ref{corollary:error}a.]

\noindent We begin with the Type 1 error expression. By Theorem \ref{thm:null_alt_distr},
\begin{align*}
    1 -  P(d F_\infty \leq z) = & 1 - G_d(z) - G^{''}_d(z)\left(z\right ) ^2 c_2b + G_d(z) z [c_1 + c_2(d-1)] b + o(b)\\
     = & 1 - G_d(z) - G^{''}_d(z)\left(z\right ) ^2 c_2b + G_d(z) z [ c_1 + c_2(d-1)] b + o(b)\\
    \overset{set}{=} & \alpha + o(b). 
\end{align*}
\noindent By Theorem \ref{thm:null_alt_distr},
\begin{align*}
    1 -  P(d F_{T, OLS} \leq z) = & 1 - G_d(z) -  G^{''}_d(z) \left(z\right ) ^2 c_2 b + G'_d(z) z[c_{b,T} + c_1 + c_2(d-1)b] \\
    & - (bT)^{-q}G'_d(z) z g_q w_q + o\left ( (bT)^{-q} \right )\\
    = & 1 - G_d(z) - G^{''}_d(z) \left( z\right ) ^2 c_2b + G'_d(z) z[c_1 + c_2(p-1)] b \\
    & + G'_d(z) z[c_{b,T}] - (bT)^{-q}G'_d(z) z g_q w_q + o\left ( (bT)^{-q} \right ).
\end{align*}
\noindent Recognizing components from Theorem \ref{thm:null_alt_distr} we can rewrite the above expression,
\begin{align*}
    1 -  P(d F_{T, OLS} \leq z) = & \alpha + G'_d(z) z[c_{b,T}]  - (bT)^{-q}G'_d(z)zg_q w_q + o(b) + o\left ( (bT)^{-q} \right ).
\end{align*}

\noindent Observing $dc^\alpha_d(b) = \chi^\alpha_d + o(b)$, and using the approximation in Remark~\ref{remark:fsb_bias} we have, 
$$e_I(b) =  \alpha + G'_d(\chi^\alpha_d)\chi^\alpha_d c_{b,T} + (bT)^{-q}G'_d(\chi^\alpha_d)\chi^\alpha_dg_q w_q.$$

\end{proof}

\begin{proof}[Corollary \ref{corollary:error}b.] 

\noindent Let $z = d c_d^\alpha(b)$ for ease of notation. Recall from Theorem \ref{thm:null_alt_distr}, 
\begin{align}
   P_\delta(dF_{OLS} \leq z) = & G_{d, \delta^2}\left (z\right) + G^{''}_{d, \delta^2}\left (z\right) \left (z\right) ^2 c_2 \nonumber\\
    & - G'_{d, \delta^2}(z) z[c_{b,T} + c_1 +  c_2 (d-1)b]  \nonumber\\
    & - (bT)^{-q}G'_{d, \delta^2}\left (z\right)zg_q w_q + o(b) + o \left( (bT)^{-q}\right).
    \label{eq:alt_simple}
\end{align}
\noindent We wish to rewrite this expression so that it has more accessible terms.  We will utilize a Cornish-Fisher type expansion as in \citet[Theorem 4]{sun2014let} and \citet[Theorem 1]{lazarus2021size}.  Cornish-Fisher expansions are used to approximate quantiles of a random variable using the first few cumulants of the Gaussian distribution \citep{holton_2016, lee1992algorithm}, provided the random variable being approximated has a Gaussian limiting distribution.  Cornish-Fisher \textit{type} expansions broaden this result to include non-Gaussian limiting distributions \citep{hill1968generalized}. For our purposes we use the $\chi^2$ limiting distribution to approximate the fixed-$b$ critical value $z$ as $b \rightarrow 0$, 
\begin{align*}
   z = d c^\alpha_d(b) = & \chi^\alpha_d - \frac{G^{''}_{d} \left(\chi_d^\alpha \right) }{G^{'}_{d} \left(\chi_d^\alpha \right)} \left (\frac{\gamma_1}{6} \right) \\
   = & \chi^\alpha_d -\frac{G^{''}_{d} \left(\chi_d^\alpha \right) }{G^{'}_{d} \left(\chi_d^\alpha \right)} \left (c_2 \left(\chi^\alpha_d\right )^2 \right) b + O(b) 
\end{align*}
\noindent where $\gamma_1$ is the skewness of the $\mathscr{F}_\infty$ random variable.  As already mentioned, $d \mathscr{F}_\infty(b,d) \rightarrow \chi^2_d$; the Cornish-Fisher type expansion above gives a structure to this approximation that helps us rearrange terms.  We plug in the expansion into the first term of (\ref{eq:alt_simple}) and further approximate with a Taylor series expansion around $\chi^2_d$,   
\begin{align*}
    G_{d, \delta^2} (z) = & G_{d, \delta^2} \left (\chi^\alpha_d -  \frac{ G^{''}_{d} (\chi^\alpha_d) c_2 \left( \chi^\alpha_d\right)^2}{G^{'}_d(\chi^\alpha_d)} b\right ) \\
    = & G_{d, \delta^2} (\chi^\alpha_d) - G^{'}_{d, \delta^2} (\chi^\alpha_d)  \left(\frac{G^{''}_d (\chi^\alpha_d) c_2 \left( \chi^\alpha_d\right)^2}{G^{'}_{d}(\chi^\alpha_d)}   \right)b +O(b) .
\end{align*}
\noindent We plug this new representation back into (\ref{eq:alt_simple}), 
\begin{align*}
    P_\delta(dF_{OLS} \leq z) = & G_{d, \delta^2} (\chi^\alpha_d) - G^{'}_{d, \delta^2} (\chi^\alpha_d)  \left(\frac{G^{''}_d (\chi^\alpha_d) c_2 \left( \chi^\alpha_d\right)^2}{G^{'}_{d}(\chi^\alpha_d)}   \right)b\\
    & + G^{''}_{d, \delta^2}\left(z\right ) \left(z\right ) ^2 c_2 - G'_{d, \delta^2}(z) z[c_{b,T} + c_1 + c_2 (d-1)b]  \\
    & - (bT)^{-q}G'_{d, \delta^2}\left(z\right )zg_q w_q + O(b) + o \left( (bT)^{-q}\right). 
\end{align*}
\noindent We further substitute $z = d c^\alpha_d(b) = \chi^\alpha_d + O(b)$ to rearrange terms, 
\begin{align*}
   P_\delta(dF_{OLS} \leq z) =  & G_{d, \delta^2} (\chi^\alpha_d) - G^{'}_{d, \delta^2} (\chi^\alpha_d)  \left(\frac{G^{''}_d (\chi^\alpha_d) c_2 \left( \chi^\alpha_d\right)^2}{G^{'}_{d}(\chi^\alpha_d)}   \right)b\\
    & + G^{''}_{d, \delta^2}(\chi^\alpha_d) (\chi^\alpha_d)^2 c_2 - G'_{d, \delta^2}(\chi^\alpha_d) \chi^\alpha_d[c_{b,T} + c_2 (d-1)b]  \\
    & - (bT)^{-q}G'_{d, \delta^2}(\chi^\alpha_d)\chi^\alpha_d g_q w_q + O(b) + o \left( (bT)^{-q}\right) .   
\end{align*}
\noindent Using the result from \citet[Theorem 5]{sun2014let},
\begin{equation*}
    G^{''}_{d, \delta^2}(\chi^\alpha_d) - G^{'}_{d, \delta^2} (\chi^\alpha_d)  \left(\frac{G^{''}_d (\chi^\alpha_d)}{G^{'}_d(\chi^\alpha_d)}   \right) = \frac{\delta^2}{2\chi^\alpha_d } G^{'}_{(d+2), \delta^2}(\chi^\alpha_d), 
\end{equation*}
\noindent we can further simplify terms 
\begin{align*}
   P_\delta(dF_{OLS} \leq  dc^\alpha_d(b)) = & G_{d, \delta^2} (\chi^\alpha_d) + \frac{\delta^2}{2} G^{'}_{(d+2), \delta^2}(\chi^\alpha_d) \chi^\alpha_d c_2 b \\
    & - G'_{d, \delta^2}(\chi^\alpha_d) \chi^\alpha_d[c_{b,T} + c_1 + c_2 (d-1)b]  \\
    & - (bT)^{-q}G'_{d, \delta^2}(\chi^\alpha_d)\chi^\alpha_d g_q w_q + O(b) + o \left( (bT)^{-q}\right).   
\end{align*}
\noindent We use Remark \ref{remark:fsb_bias} to reach our final expression: 

\begin{align*}
    e_{II}(b)= & G_{d, \delta^2} (\chi^\alpha_d) + \frac{\delta^2}{2} G^{'}_{(d+2), \delta^2}(\chi^\alpha_d) \chi^\alpha_d c_2 b  - G'_{d, \delta^2}(\chi^\alpha_d) \chi^\alpha_d \frac{2\rho^2}{1+\rho}\rho^{bT}\\
    & - G'_{d, \delta^2}(\chi^\alpha_d) \chi^\alpha_d[c_1 + c_2 (d-1)b] - (bT)^{-q}G'_{d, \delta^2}   (\chi^\alpha_d)\chi^\alpha_d g_q w_q.
\end{align*}

\end{proof}


\end{appendix}


\section*{Acknowledgments}
The authors gratefully acknowledge the anonymous referees, the Associate Editor, and the Editor for their thoughtful and constructive comments, which have significantly improved the quality of this paper.

\bibliographystyle{plainnat} 
\bibliography{references}       

\end{document}